%% file: main.tex
\documentclass[a4paper,10pt]{article}
\usepackage[T1]{fontenc}
\usepackage[utf8]{inputenc}
\usepackage[english]{babel}
\usepackage{geometry}
\usepackage{lipsum}
\usepackage[hidelinks]{hyperref}
\usepackage{xcolor}
\usepackage{microtype}
\usepackage{amsmath, amsthm, amssymb, mathtools}
\usepackage{siunitx}
\usepackage{resizegather}
\usepackage{bm, bbm, mathrsfs}
\usepackage{pgfplots, pgfplotstable, booktabs, array}
	\pgfplotsset{compat=1.17}
	\pgfplotstableset{
		every head row/.append style=%
			{before row=\toprule, after row=\midrule},
		every last row/.append style=%
			{after row=\bottomrule}
	}
\usepackage{caption, subcaption}
\renewcommand{\vec}[1]{\bm{#1}}

\newcommand{\N}{\mathbb{N}}

\newcommand{\R}{\mathbb{R}}

\newcommand{\abs}[1]{\left\lvert#1\right\rvert}

\newcommand{\norm}[1]{\left\lVert#1\right\rVert}
\newcommand{\normbig}[1]{\bigl\lVert#1\bigr\rVert}
\newcommand{\normBig}[1]{\Bigl\lVert#1\Bigr\rVert}
\newcommand{\normbigg}[1]{\biggl\lVert#1\biggr\rVert}
\newcommand{\normBigg}[1]{\Biggl\lVert#1\Biggr\rVert}

\newcommand{\restrict}[1]{\hspace*{0.05em}
	\raisebox{-0.25ex}{\resizebox{!}{0.8\height}{$\mid$}${}_{#1}$}}
\newcommand{\deq}{\vcentcolon=}

\newcommand{\dx}{\, dx}
\newcommand{\dy}{\, dy}

\newcommand{\deta}{\, d\eta}
\newcommand{\dsigma}{\,d\sigma}

\DeclareMathOperator{\diver}{div}
\DeclareMathOperator{\diag}{diag}

\theoremstyle{plain}
\newtheorem{lemma}{Lemma}[section]
\newtheorem{theorem}[lemma]{Theorem}
\newtheorem{proposition}[lemma]{Proposition}

\theoremstyle{definition}
\newtheorem{definition}[lemma]{Definition}

\newtheorem{example}[lemma]{Example}
\allowdisplaybreaks
\newcommand{\Omegabar}{\overline{\Omega}}

\newcommand{\Fcal}{\mathcal{F}}
\newcommand{\Ical}{\mathcal{I}}
\newcommand{\Kcal}{\mathcal{K}}
\newcommand{\Lcal}{\mathcal{L}}
\newcommand{\Qcal}{\mathcal{Q}}
\newcommand{\Tcal}{\mathcal{T}}

\newcommand{\uhat}{\hat{u}}
\newcommand{\uhatvec}{\hat{\vec{u}}}

\newcommand{\h}{\vec{h}}
\newcommand{\w}{\vec{w}}

\begin{document}

\title{A decoupled meshless Nyström scheme for 2D
Fredholm\\ integral equations of the second kind
with smooth kernels}
\author{
Bruno Degli Esposti\thanks{Department of Information
    Engineering and Mathematics, University of Siena,
	Via Banchi di Sotto 55, 53100 Siena, Italy,
	\url{bruno.degliesposti@unifi.it}}, \hspace*{4pt}
Alessandra Sestini\thanks{Department of Mathematics
	and Computer Science, University of Florence,
	Viale Morgagni 67, 50134 Firenze, Italy,
	\url{alessandra.sestini@unifi.it}}
}
\date{July 23, 2026}
\maketitle

\begin{abstract}
The Nyström method for the numerical solution
of Fredholm integral equations of the
second kind is generalized by decoupling
the set of solution nodes
from the set of quadrature nodes.
The accuracy and efficiency of the new
method is investigated for smooth
kernels and complex 2D domains
using recently developed meshless moment-free
quadrature formulas on scattered nodes.
Compared to the classical Nyström method,
our variant has a clear performance
advantage, especially for narrow kernels.
The decoupled Nyström method requires
the choice of a reconstruction scheme
to approximate values at quadrature nodes
from values at solution nodes.
We prove that, under natural assumptions,
the overall order of convergence is the minimum
between that of the quadrature scheme
and of the reconstruction scheme.
When the two schemes have the same order,
we describe a simple procedure to determine experimentally
 a nearly optimal ratio between the
densities of solution and quadrature nodes,
for which our decoupled Nyström method achieves
its maximum efficiency.
\end{abstract}

\input{section1}
\input{section2}
\input{section3}
\input{section4}
\input{section5}

\section*{Acknowledgements}
The authors acknowledge the contribution of the National Recovery
and Resilience Plan, Mission 4 Component 2 - Investment 1.4 - NATIONAL
CENTER FOR HPC, BIG DATA AND QUANTUM COMPUTING - funded by the
European Union - NextGenerationEU - (CUP B83C22002830001).
The authors are members of the
Gruppo Nazionale Calcolo Scientifico - Istituto
Nazionale di Alta Matematica (GNCS-INdAM).
The INdAM support through GNCS 2025 project with CUP E53C24001950001
and GNCS 2026 project with CUP E53C25002010001 is gratefully acknowledged. 

\appendix
\numberwithin{equation}{section}
\input{appendixA}
\input{appendixB}

\bibliographystyle{abbrv}
\bibliography{bibliography}

\end{document}

%% file: section1.tex
\graphicspath{{./figures/}}

\section{Introduction}
\label{sec:introduction}
Fredholm Integral Equations (FIEs) of the second
kind frequently arise in modeling various phenomena,
such as fractures in continuum mechanics
\cite{silling2000reformulation},
light transport (radiance integral equation)
in computer graphics
\cite{kajiya1986rendering,keller1997instant},
microscale biological flows using the
method of regularized stokeslets
\cite{cortez2005method},
equilibrium states of nonlocal evolution
equations arising in materials science
\cite{Bates06,duncan2000coarsening}
and mathematical biology
\cite{murray2011mathematical1}, in particular modeling
population dispersal \cite{hutson2003evolution}
and pattern formation
\cite{kondo2017updated,murray2011mathematical2}.

For this reason and also for its intrinsic mathematical beauty,
the numerical treatment of FIEs has attracted considerable interest in
the literature for several decades, relying on two main kinds of
discretizations: projection methods such as collocation or Galerkin,
and the Nyström method; see \cite{atkinson1997numerical} for a general
and rigorous introduction to the topic. Both collocation
and Galerkin approaches require the preliminary selection of a
finite-dimensional space where the adopted orthogonal or interpolatory
operator defines a functional projection, see for example
\cite{Kulkarni25} for a recent introduction dealing with piecewise
polynomial spaces and also introducing the more effective modified
projection method in the 1D setting. In the same context, two recent
proposals also improve on the projection process, the first
relying on spline quasi-interpolation and focusing on weakly singular
kernels \cite{ALR24} and the second based on constrained
least squares mock-Chebyshev polynomials \cite{DMNO24}.

The popularity of the Nyström method is similar if not even greater
than that of projection methods, since it does not require
the preliminary selection of any approximating space,
while still producing a functional numerical solution of the FIE,
just relying on a suitable quadrature rule.
In particular, due to their effectiveness,
Gaussian quadrature rules are often considered
for formulating the 1D Nyström method, as well as their
tensor product extension to the multivariate setting when
the spatial domain is a hyperrectangle.
In this way high-order methods can also be defined.
An interesting recent variant that ensures an even further
improvement in the attainable accuracy has been considered
in \cite{BBCR22} for 1D domains, where some Gaussian
quadrature rules are proposed for low-degree spline spaces
with uniform breakpoints, together with a recursive algorithm
that generates the set of quadrature rule knots and weights
which need to be computed, being dependent on the
considered spline space.

Since much has already been done for the Nyström method
in the 1D setting, its 2D extension is attracting increasing
attention; see, for example, \cite{laguardia2023nystrom}
and references therein, and also \cite{mezzanotte24}, 
where several methods relying on radial basis functions are reviewed.
As an alternative, for general 2D domains,
we combine in this paper the Nyström method with the
high-order meshless quadrature formulas developed and analyzed
in \cite{davydov2025meshless}, see also the doctoral
thesis \cite{degliesposti2025domain}. We postpone
comparisons with existing meshless Nyström schemes
to Section~\ref{ssec:decoupled-meshless}, following an in-depth
explanation of our approach. The same section also
describes how sets of unstructured nodes can be generated
from a parametric description of the boundary of a 2D domain
with an advancing front method.

Our main contribution in this paper is the novel decoupling
of solution nodes and quadrature nodes in a Nyström scheme,
which we have developed in a high-order, meshless setting.
This generalization of the classical Nyström method requires
the introduction of a suitable \emph{reconstruction scheme},
and the error analysis of Theorem~\ref{theo:nystrom-decoupled}
justifies the interpretation of the overall numerical error
as the sum of two kinds of errors: a reconstruction error
that depends on the density of solution nodes, and
a quadrature error that depends on the density of quadrature nodes.
Especially for FIEs with narrow kernels, such as Gaussian kernels
with small standard deviation, the numerical solution can be
found much more efficiently by using a set of quadrature nodes
that is significantly finer than the set of solution nodes.
All of our theoretical results are supported by suitable
numerical experiments, which are
presented in Section~\ref{sec:numerical-experiments}.

Finally, it deserves to be mentioned that some boundary integral
equations derived as integral reformulation of elliptic
boundary value problems can be numerically treated by the Nyström
method, since they are also Fredholm integral equations,
as mentioned by the authors and their co-authors
in~\cite{IgA3D}, where the numerical treatment of 3D Helmholtz
problems is considered in the context of an Isogeometric
Boundary Element Method (IgA-BEM). However, in such
integral equations, the kernel is necessarily singular,
an assumption which goes beyond the scope of the present paper.
Indeed, specific quadrature rules suited for singular
and near-singular integrands are needed to achieve
satisfactory accuracy or even just convergence of the numerical scheme.
An approach in this direction is~\cite{Marussig16},
where boundary integral equations are considered on a
multi-patch isogeometric boundary representation of a 3D domain
and the regularization of singular integrals
is performed by suitable local corrections.

%% file: section2.tex
\graphicspath{{./figures/}}

\section{Fredholm integral equations of the second kind}
\label{sec:fie}
A \emph{Fredholm integral equation}
on a bounded domain $\Omega \subset \R^d$
is the integral equation
\begin{equation} \label{eq:fie}
\lambda u(x) - \int_{\Omega} k(x,y) u(y) \dy = f(x)
\end{equation}
for the unknown $u \colon \Omega \to \R$,
scalar parameter $\lambda \in \R$,
absolutely integrable kernel
$k \colon \Omega \times \Omega \to \R$
and right-hand side $f \colon \Omega \to \R$.
In this work, we assume with no exceptions
that $\lambda \neq 0$,
that $k \in C^0(\Omegabar \times \Omegabar)$,
and that $f$ is a nonzero function
in $C^0(\Omegabar)$. Under these assumptions,
equation \eqref{eq:fie} is called a \emph{nonhomogeneous
Fredholm integral equation of the second kind
with continuous kernel}, and its well-posedness
can be readily investigated using tools
from functional analysis.
Note that, even though the assumption of continuous kernel
is restrictive for some applications, all the mathematical
models cited in the introduction of this paper are based
on smooth kernels.

Now the space $C^0(\Omegabar)$ is a Banach space
with respect to the $\infty$-norm
\[
\norm{u}_\infty \deq \max_{x \in \Omegabar} \abs{u(x)},
\quad u \in C^0(\Omegabar),
\]
and it can be proved that integration against
the kernel $k$ defines a bounded linear operator as follows
\[
\Kcal \colon C^0(\Omegabar) \to C^0(\Omegabar), \quad
(\Kcal u)(x) \deq \int_{\Omega} k(x,y) u(y) \dy, \quad
\norm{\Kcal} = \max_{x \in \Omegabar}
    \int_\Omega \abs{k(x,y)} \dy.
\]
Then, using this notation, the  integral equation \eqref{eq:fie}
can be expressed in functional notation as
\[
(\lambda \Ical - \Kcal) u = f,
\]
with $\Ical$ being the identity operator on $C^0(\Omegabar)$.
Crucially, the image $\Kcal B$ of the closed unit ball
\[
B = \left\{ u \in C^0(\Omegabar)
\mid \norm{u}_\infty \leq 1 \right\}
\]
is not only bounded, but also uniformly equicontinuous.
By the Arzelà-Ascoli theorem
(see Proposition~\ref{prop:asc-arz} in
Appendix~\ref{sec:appendix-functional-analysis}), it follows
that $\Kcal B$ is relatively compact, and so
$\Kcal$ is a compact operator.
Then, the Fredholm alternative implies that exactly
one of the following statements must hold:
\begin{itemize}
\item There is a non-trivial solution
$u \in C^0(\Omegabar)$ to the eigenvalue problem
$\Kcal u = \lambda u$.
\item The operator $\lambda \Ical - \Kcal$ has
a bounded inverse.
\end{itemize}
In other words, the condition $\lambda \not\in \sigma(\Kcal)$,
with $\sigma(\Kcal)$ being the spectrum of the integral
operator $\Kcal$, is a necessary and sufficient condition
for equation \eqref{eq:fie} to be uniquely solvable
in the space $C^0(\Omegabar)$,
with $u$ depending continuously on $f$.
Since the spectral radius of $\Kcal$ is bounded
from above by the norm of $\Kcal$, a sufficient
and easily verifiable condition on $\lambda$
to ensure well-posedness of \eqref{eq:fie}
is $\abs{\lambda} > \norm{\Kcal}$.

For the sake of developing high-order
discretizations of \eqref{eq:fie}, however,
regularity of the solutions must be
established beyond simple continuity.
Let $C^q(\Omegabar)$ be the space of functions
over $\Omega$ whose partial derivatives up to
integer order $q \geq 1$ exist and can be
continuously extended to $\Omegabar$, see
Appendix~\ref{sec:appendix-functional-analysis}
for a rigorous definition.
Then, assuming appropriate regularity
of the kernel and right-hand side,
\[
k \in C^q(\Omegabar \times \Omegabar),
\quad f \in C^q(\Omegabar),
\]
the inclusion $u \in C^q(\Omegabar)$ can
be proved by differentiating \eqref{eq:fie}
under the integral sign, because $\lambda \neq 0$.

Since $C^q(\Omegabar)$ is also a Banach
space with respect to a suitable generalization
of the $\infty$-norm involving derivatives up
to order $q$ (see
Appendix~\ref{sec:appendix-functional-analysis}
for more details),
$\Kcal$ can be either understood as an
integral operator over $C^0(\Omegabar)$
or over $C^q(\Omegabar)$ with $q \geq 1$,
depending on context.
We conclude this section by defining some
properties of kernels.
A kernel $k \colon \Omega \times \Omega \to \R$
is said to be
\begin{itemize}
\item \emph{strictly positive}, if $k(x,y) > 0$
for all $x,y \in \Omega$, and \emph{nonnegative}
if $k(x,y) \geq 0$.
\item \emph{symmetric}, if $k(x,y) = k(y,x)$
for all $x,y \in \Omega$.
\item \emph{translation-invariant},
if there exists $J \colon \R^d \to \R$
such that $k(x,y) = J(x-y)$
for all $x,y \in \Omega$.
\item \emph{radial},
if there exists
$\varphi \colon \R^+ \to \R$
such that $k(x,y) = \varphi\bigl(\norm{x-y}\bigr)$
for all $x,y \in \Omega$.
\end{itemize}
Infinitely smooth, strictly positive, radial kernels
are the most common in applications,
with a prominent example being the
\emph{Gaussian kernel}
\begin{equation} \label{eq:gaussian}
k(x,y) = \bigl(\sigma \sqrt{2\pi}\bigr)^{-d}
\exp\bigl(-\norm{x-y}^2/(2\sigma^2)\bigr),
\quad x,y \in \R^d
\end{equation}
with standard deviation $\sigma > 0$.
The normalization factor
$\bigl(\sigma \sqrt{2\pi}\bigr)^{-d}$
ensures that
\[
\int_{\R^d} k(x,y) \dy = 1
\quad
\text{for all $x \in \R^d$ and all $\sigma > 0$}.
\]
An immediate consequence is that the operator
$\Kcal \colon C^0(\Omegabar) \to C^0(\Omegabar)$
associated with \eqref{eq:gaussian} satisfies
$\norm{\Kcal} \leq 1$, and so the Fredholm integral
equation \eqref{eq:fie} is well-posed for any $\lambda > 1$.

%% file: section3.tex
\graphicspath{{./figures/}}

\section{The Nyström method}
\label{sec:numerical-scheme}
The numerical solution of a nonhomogeneous
Fredholm integral equation of the second kind
with regular kernel such as \eqref{eq:fie}
can be computed using a wide range of techniques, see
\cite{atkinson1997numerical,hackbusch1995integral}
for classical treatments of the subject.
The main techniques are the degenerate kernel method,
projection methods (such as collocation and Galerkin),
and the Nyström method, also known as the \emph{quadrature}
method, because it amounts to discretizing the integral
in \eqref{eq:fie} directly with a quadrature formula.
In this paper, we present a generalization of Nyström's method
that allows the quadrature formula to be chosen independently
of the points where the exact solution is approximated,
and demonstrate its advantages in a multivariate, meshless setting.
Before introducing our generalization,
we briefly recall some key definitions
regarding quadrature schemes,
and revisit the classical Nyström method.

\subsection{Quadrature schemes and the classical Nyström method}
\label{ssec:classical-nystrom}
Let $\{(Y_h,\w_h)\}_{h>0}$ be
a \emph{quadrature scheme} on $\Omega$,
i.e.~a family of quadrature formulas with nodes
\[
Y_h = \bigl\{y_{h,i} \mid i = 1,\dots,\abs{Y_h}\bigr\}
\subset \Omegabar
\]
and weights $\w_h \in \R^{\abs{Y_h}}$
indexed by a parameter $h > 0$
such that $\lim_{h \to 0^+} \abs{Y_h} = \infty$.
A more precise definition of $h$ is not
needed to present and prove the main
theoretical results in this paper.
Instead, using $h$ as a general placeholder
makes our analysis applicable to
a wider range of situations.
Nevertheless, to give intuition for $h$
in practical contexts, we briefly review
several common definitions.

In the context of composite
quadrature formulas defined using a mesh
over $\Omegabar$, the parameter $h$
usually corresponds to the diameter of the largest
element in the mesh. When dealing with
meshless methods on scattered data,
$h$ is typically related to the fill distance
of the set $Y_h$, defined as the radius of
the largest ball centered in $\Omega$ which
does not include elements of $Y_h$.
Another common definition is that
of packing distance
$(\abs{\Omega}/\abs{Y_h})^{1/d}$.
Finally, $h$ could be simply an input parameter
of a node generation algorithm, such as the
advancing front method described in
Section~\ref{ssec:decoupled-meshless}.

Regardless of how $h$ is defined, the
convergence order and asymptotic stability
of a quadrature scheme can be characterized
as follows. For simplicity, we restrict
our attention to integer convergence orders.
\begin{definition} \label{Cw}
A quadrature scheme $\{(Y_h,\w_h)\}_{h>0}$
is said to be \emph{convergent with
order $q \in \N$, $q > 0$} if there exists a constant $C_w$
such that, for all $h > 0$,
\[
\abs{
    \int_{\Omega} v(x) \dx
    - \sum_{i=1}^{\abs{Y_h}} w_{h,i} v(y_{h,i})
}
\leq C_w \, h^q \norm{v}_{C^q(\Omegabar)}
\quad \text{for all $v \in C^q(\Omegabar)$}.
\]
\end{definition}
It can be proved that the nodes $Y_h$ of a convergent
quadrature scheme have to fill up $\Omega$ for smaller
and smaller values of $h$, in the sense that the fill
distance of $Y_h$ must go to zero for $h \to 0^+$.
\begin{definition} \label{CQ}
A quadrature scheme is said to be \emph{stable} if there exists a constant $C_Q$ such that
\[
\norm{\w_h}_1 \leq C_Q \quad \text{for all $h>0$}.
\]
\end{definition}
Note that a quadrature scheme convergent
for functions in $C^q(\Omegabar)$ might not
be convergent for functions of lower regularity.
In particular, it might not be convergent
for functions that are only continuous,
as shown by Example~\ref{ex:unstableqf}
in the appendix.
However, such counterexamples are mostly artificial
and purely of theoretical interest.
Unstable schemes are always avoided in
computational practice, for example because
they are unreliable for noisy data.
Since convergence for continuous functions
will be important in our theoretical analysis
of the Nyström method, we have included
in Appendix~\ref{sec:appendix-quadrature}
a proof that, for a quadrature scheme with
convergence order $q > 0$, stability
is equivalent to convergence
for continuous functions.

Moving on to the Nyström method,
discretizing the integral in \eqref{eq:fie}
with a quadrature scheme leads to
\begin{equation} \label{eq:semiscretization}
\lambda u_h(x)
- \sum_{i=1}^{\abs{Y_h}} w_{h,i} k(x,y_{h,i}) u_h(y_{h,i})
= f(x)
\qquad \text{for all $x \in \Omega$},
\end{equation}
a functional equation for the new unknown
$u_h \in C^q(\Omegabar)$, whose importance
will be explained later.
We can rewrite \eqref{eq:semiscretization}
more compactly as the semidiscrete operator equation
\begin{equation} \label{eq:semiOperator}
(\lambda \Ical - \Kcal_h) u_h = f,
\end{equation}
with $\Kcal_h$ being the discretization
of $\Kcal$ introduced by the chosen quadrature scheme:
\[
\Kcal_h \colon C^q(\Omegabar) \to C^q(\Omegabar)
\qquad (\Kcal_h v)(x) \deq
\sum_{i=1}^{\abs{Y_h}} w_{h,i} k(x,y_{h,i}) v(y_{h,i}).
\]
Imposing exactness of \eqref{eq:semiscretization}
at the discrete set of nodes $x \in Y_h$
leads to the square system of equations
\begin{equation} \label{eq:nystrom-classical-scalar}
\lambda \uhat_{h,i}
- \sum_{j=1}^{\abs{Y_h}} w_{h,j} k(y_{h,i},y_{h,j}) \uhat_{h,j}
= f(y_{h,i}), \qquad i = 1,\dots,\abs{Y_h},
\end{equation}
for the unknowns
$\uhat_{h,i}$,
which are meant to approximate the exact
values $u(y_{h,i})$. This is the
\emph{classical Nyström method}
for the numerical solution of \eqref{eq:fie}.
By introducing $I_h$ as the identity
matrix of size $\abs{Y_h} \times \abs{Y_h}$,
and the matrices
\[
(K_h)_{ij} = k(y_{h,i},y_{h,j}),
\quad W_h = \diag{(\w_h)},
\]
we can rewrite the linear system
\eqref{eq:nystrom-classical-scalar}
more compactly as
\begin{equation} \label{eq:nystrom-classical}
(\lambda I_h - K_h W_h) \uhatvec_h = f \restrict{Y_h}.
\end{equation}
A remarkable property of the Nyström method
is that, once the values of the solution
of \eqref{eq:semiscretization} are known
at all points in $Y_h$, the value at any other point
can be computed by the simple expression
\begin{equation} \label{Nistrominterp}
u_h(x) = \frac{1}{\lambda} \biggl(
\sum_{i=1}^{\abs{Y_h}} w_{h,i} k(x,y_{h,i})
u_h(y_{h,i}) + f(x) \biggr),
\end{equation}
known as \emph{Nyström interpolation formula}.
Since the values of $u_h$ at $Y_h$ must
satisfy the linear system
\eqref{eq:nystrom-classical},
the functional equation \eqref{eq:semiscretization}
has a unique solution $u_h$ if and only
if \eqref{eq:nystrom-classical}
has a unique solution $\uhatvec_h$,
in which case the identity
$u_h(y_{h,i}) = \uhat_{h,i}$
holds for all $i = 1,\dots,\abs{Y_h}$.

The existence of a unique solution
to system \eqref{eq:nystrom-classical}
and the overall convergence of the
Nyström interpolant $u_h$ to the
exact solution $u$ of \eqref{eq:fie}
can be proved under reasonable assumptions
on $\lambda$, on the quadrature scheme,
and on the smoothness of $k$ and $f$.
We recall here the classical approach
followed in~\cite{atkinson1997numerical},
which we have extended with a quantitative
estimate of the order of convergence.

\begin{proposition} \label{prop:nystrom-classical}
With reference to the classical Nyström
method \eqref{eq:nystrom-classical}
applied to the Fredholm integral equation
of the second kind \eqref{eq:fie}, suppose that
\begin{enumerate}
\item The kernel $k$ belongs to
$C^q(\Omegabar \times \Omegabar)$,
and the right-hand side
$f$ belongs to $C^q(\Omegabar)$.
\item The scalar parameter $\lambda$ does not
belong to the spectrum of
$\Kcal \colon C^0(\Omegabar) \to C^0(\Omegabar)$.
\item The quadrature scheme $\{(Y_h,\w_h)\}_{h>0}$
has order $q > 0$ and is stable.
\end{enumerate}
Then, for all sufficiently small $h > 0$,
the linear system \eqref{eq:nystrom-classical}
has a unique solution, the condition number
of the system matrix is bounded uniformly
with respect to $h$, and there exists a constant
$C > 0$ independent of $h$ and the exact solution
$u \in C^q(\Omegabar)$ such that
\[
\norm{u - u_h}_{\infty}
\leq C h^q \norm{u}_{C^q(\Omegabar)}.
\]
\end{proposition}

\begin{proof}
By Proposition~\ref{prop:convcont} reported in Appendix~\ref{sec:appendix-quadrature},
the third assumption implies that the quadrature
scheme is convergent for all functions in $C^0(\Omegabar)$. This in turn makes
Theorem~4.1.2 of \cite{atkinson1997numerical}
applicable to our setting, from which it follows
that, for all sufficiently small $h > 0$,
the operator $(\lambda \Ical - \Kcal_h) \colon
C^0(\Omegabar) \rightarrow C^0(\Omegabar)$
is invertible and there exists
a constant $C_s$ independent of $h$ such that
\[
\norm{(\lambda \Ical - \Kcal_h)^{-1}} \leq C_s.
\]
Moreover, the linear
system~\eqref{eq:nystrom-classical} has a unique
solution. The analysis at the end of
Section~4.1.1 of \cite{atkinson1997numerical}
proves the stated result about the
condition number of the system matrix. Now,
\[
\Kcal u -\Kcal_h u
= \lambda u - f -\Kcal_h u
= \lambda u - \lambda u_h
    + \Kcal_h u_h -\Kcal_h u
= (\lambda \Ical - \Kcal_h)(u-u_h).
\]
Then, from the previous inequality,
we can derive the bound
\[
\Vert u - u_h \Vert_\infty
= \Vert (\lambda \Ical - \Kcal_h)^{-1} (\Kcal u -\Kcal_h u) \Vert_\infty
\le C_s \Vert (\Kcal  -\Kcal_h) u\Vert_\infty\,.
\]
By the third assumption and
Proposition~\ref{prop:submult} reported in Appendix~\ref{sec:appendix-functional-analysis},
there exists a constant $C_w$ such that
\[
\norm{(\Kcal  -\Kcal_h)u}_\infty
\leq \max_{x \in \Omegabar} \left\{
    C_w \, h^q
    \norm{k(x,\cdot) u(\cdot)}_{C^q(\Omegabar)}
    \right\}
\leq C_w \, h^q \norm{u}_{C^q(\Omegabar)}
    \max_{x \in \Omegabar} \left\{
    \norm{k(x,\cdot)}_{C^q(\Omegabar)} \right\}.
\]
Hence, setting
$C = C_s C_w \max_{x \in \Omegabar}
\norm{k(x,\cdot)}_{C^q(\Omegabar)}$
completes the proof.
\end{proof}

\subsection{A decoupled variant of the Nyström method}
\label{ssec:decoupled-nystrom}
Although imposing exactness of
\eqref{eq:semiscretization}
at the quadrature nodes $Y_h$ is a perfectly
natural choice, it has the undesirable consequence
of tying the size of the system matrix $A_h$
to $\abs{Y_h}$.
This means that, to avoid unreasonably large
linear systems and maintain a good ratio
of accuracy to computational effort compared
to other methods, the quadrature formulas must
converge very quickly as $\abs{Y_h} \to \infty$.

This is not a problem on one-dimensional domains,
where e.g.~Gauss–Legendre formulas can be used,
or on simple multivariate domains that can be
decomposed into a small number of hypercubes,
possibly composed with smooth mappings.
In that case, tensor-product formulas
can be easily constructed on each hypercube,
and an exponential rate of convergence
is achieved if the integrand is sufficiently smooth.

On more complex multivariate domains, however,
a significantly larger number of quadrature nodes
is typically needed to reach the same accuracy
as Gaussian formulas on hypercubes.
This happens, for example, in the case
of composite quadrature formulas on simplicial
meshes with a fixed algebraic rate of convergence,
like the scheme described in Section~5.2.3
of \cite{atkinson1997numerical}.
In such settings, it is harder for the Nyström method
to remain competitive with e.g.~collocation methods
using splines on triangulations,
under the assumption that the cost of solving
a dense linear system with $\abs{Y_h}$ unknowns
dominates the cost of constructing and evaluating
a quadrature formula of the same size.
This is one of the reasons behind recent efforts
to achieve a spectral rate of convergence
on curvilinear domains by combining
the algebraic quadrature formulas described
in~\cite{santin2011algebraic}
with the classical Nyström method,
see~\cite{laguardia2023nystrom}.

In the case of collocation or Galerkin methods,
the size of a (square) system matrix depends only on the
dimension of the functional space that is used to
approximate the exact solution (trial space),
and not on the quadrature formulas
used to approximate the integrals.
An important consequence in applications is that
the density of quadrature nodes can be
locally refined to handle
integration of singular kernels $k(x,y)$,
or kernels that, despite being smooth, have very
large derivatives, such as Gaussian kernels
\eqref{eq:gaussian} for small values of $\sigma$.
This approach is fundamentally incompatible with
the classical Nyström method: refinement of $Y_h$
would only generate more functions to be integrated,
each one centered around a different quadrature node in $Y_h$.

Overall, there are clear reasons why one might look
for a generalization of the Nyström method where
the size of the system matrix is decoupled from the
number of quadrature nodes.
To this aim, we introduce a new family of
nodes $\{X_h\}_{h>0}$, which we call
\emph{solution nodes},
and modify the classical Nyström method
by imposing exactness of the functional
equation \eqref{eq:semiscretization}
at the discrete set of solution nodes
$x_{h,i} \in X_h$ instead of quadrature nodes $Y_h$:
\begin{equation} \label{eq:nonsq}
\lambda u_h(x_{h,i})
- \sum_{j=1}^{\abs{Y_h}} w_{h,j}
    k(x_{h,i},y_{h,j}) u_h(y_{h,j})
= f(x_{h,i})
\qquad \text{for all $i = 1,\dots,\abs{X_h}$}.
\end{equation}
What we obtain, however, is not a
square linear system as before. Indeed, assuming
$X_h \neq Y_h$, the number of unknowns
$\abs{X_h \cup Y_h}$ is strictly larger than
the number of equations $\abs{X_h}$.
A natural way to bridge the gap is to
approximate the unknowns $u_h(Y_h)$
from the values $u_h(X_h)$ using some
interpolation or quasi-interpolation scheme.
For any family of matrices $\{R_h\}_{h>0}$
of size $\abs{Y_h} \times \abs{X_h}$ such that
\[
v_{|Y_h} \approx R_h v_{|X_h}
\quad \text{for all $v \in C^q(\Omegabar)$},
\]
we can use $R_h$ to eliminate the unknowns
$u_h(Y_h)$ from \eqref{eq:nonsq}
and obtain the square system
\begin{equation} \label{eq:nystrom-decoupled-scalar}
\lambda \uhat_{h,i}
- \sum_{j=1}^{\abs{Y_h}} w_{h,j} k(x_{h,i},y_{h,j})
  \sum_{k=1}^{\abs{X_h}} (R_h)_{jk} \uhat_{h,k}
= f(x_{h,i})
\qquad \text{for all $i = 1,\dots,\abs{X_h}$}.
\end{equation}
This is, at its core, the main idea
behind our new decoupled Nyström method.
Before moving on to a formal error analysis
of this method, however, we introduce a further
generalization by allowing the families
$\{X_h\}$ and $\{Y_h\}$
to be indexed by two separate
parameters $h_X > 0$ and $h_Y > 0$.
In this way, the effect
of refinement of solution nodes
can be studied separately from that
of quadrature nodes.

To keep notation as simple as possible,
all the $h$-indexed quantities introduced
so far will be replaced in the following
by $\h$-indexed ones, with $\h$ being
the vector of new parameters $(h_X,h_Y)$.
As explained at the beginning of
Section~\ref{ssec:classical-nystrom},
$h_X$ and $h_Y$ are just placeholders,
although they could be defined
more precisely as e.g.~the
fill distances of the sets of
solution nodes and quadrature nodes,
respectively. Although the notation
$X_{\h}$ might suggest that $X_{\h}$
also depends on $h_Y$, we stress that
$X_{\h}$ only depends on $h_X$, and
that $Y_{\h}$ only depends on $h_Y$.
Some quantities, however, such as
the matrices $R_{\h}$, genuinely
depend on both $h_X$ and $h_Y$,
as formalized by the following definition.

\begin{definition}
A \emph{reconstruction scheme} is
a family of triples
\[
\bigl\{ (X_{\h},Y_{\h},R_{\h}) \mid \h = (h_X,h_Y), 
\,\, h_X > 0, \, h_Y > 0 \bigr\},
\]
with $X_{\h}$ and $Y_{\h}$ being finite subsets
of $\Omegabar$, and $R_{\h}$ being real matrices
of size $\abs{Y_{\h}} \times \abs{X_{\h}}$.
\end{definition}

Concrete examples of reconstruction schemes will be given in
Section~\ref{ssec:meshless-quadrature-reconstruction},
with particular emphasis on techniques based
on radial basis function (RBF) interpolation.
The convergence order and asymptotic stability
of a reconstruction scheme are defined
similarly to quadrature schemes.

\begin{definition} \label{CR}
A reconstruction scheme
$\{(X_{\h},Y_{\h},R_{\h}) \mid h_X > 0, \, h_Y > 0\}$
is convergent with order $q>0$
if there exists a constant $C_R$ such that,
for all $h_X > 0$ and $h_Y > 0$,
\[
\normbig{v_{|Y_{\h}} - R_{\h} v_{|X_{\h}}}_{\infty}
\leq C_R \, h_X^q \normbig{v}_{C^q(\Omegabar)}
\quad \text{for all $v \in C^q(\Omegabar)$.}
\]
\end{definition}

\begin{definition} \label{CI}
A reconstruction scheme
$\{(X_{\h},Y_{\h},R_{\h}) \mid h_X > 0, \, h_Y > 0\}$
is stable if there exists a constant $C_I$
such that
\[
\norm{R_{\h}}_\infty \leq C_I
\quad \text{for all $h_X > 0$ and $h_Y > 0$}.
\]
\end{definition}

It can be proved that, if the fill distance of the
nodes $Y_{\h}$ of a convergent reconstruction scheme
goes to zero for $\h \to 0^+$, the fill distance
of nodes $X_{\h}$ must also converge to zero.

Going back to the decoupled Nyström method, we can rewrite
\eqref{eq:nystrom-decoupled-scalar} more compactly as
\begin{equation}
\label{eq:nystrom-decoupled}
A_{\h} \uhatvec_{\h}
\deq (\lambda I_{\h} - K_{\h} W_{\h} R_{\h})
\uhatvec_{\h}
= f \restrict{X_{\h}}
\end{equation}
by denoting the identity matrix of size
$\abs{X_{\h}} \times \abs{X_{\h}}$ as $I_{\h}$,
and introducing the matrices
\[
(K_{\h})_{ij} = k(x_{\h,i},y_{\h,j}),
\quad W_{\h} = \diag{(\w_{\h})}.
\]
Note that the kernel matrix $K_{\h}$ has size
$\abs{X_{\h}} \times \abs{Y_{\h}}$,
and the reconstruction matrix $R_{\h}$ has size
$\abs{Y_{\h}} \times \abs{X_{\h}}$,
so that the product $K_{\h} W_{\h} R_{\h}$
is a square matrix.
The size of the system matrix $A_{\h}$,
being equal to the number of solution nodes in $X_{\h}$,
is now decoupled from the number of quadrature
nodes in $Y_{\h}$.
This means that the value of $h_X$ relative to $h_Y$
is a free parameter that we can choose to improve
the efficiency of Nyström's method.

To make it easier to quantify the relative size
of the sets $X_{\h}$ and $Y_{\h}$, we assume
for the rest of the paper that the \emph{decoupling ratio}
$h_X/h_Y$ is a constant, which we denote by $\gamma$.
In this way, assuming that nodes in $X_{\h}$ and $Y_{\h}$
are distributed in a nearly uniform way with average spacing
$h_X$ and $h_Y$ between adjacent nodes,
\[
\lim_{h \to 0^+} \abs{Y_h}/\abs{X_h} = \gamma^d.
\]
Consider the $\h$-indexed family of linear operators
$\Kcal_{\h} \colon C^0(\Omegabar) \to C^0(\Omegabar)$
defined by
\begin{equation} \label{eq:Kh}
(\Kcal_{\h} v)(x)
= \sum_{i=1}^{\abs{Y_{\h}}} \w_{\h,i} k(x,y_{\h,i})
    \sum_{j=1}^{\abs{X_{\h}}} (R_{\h})_{ij} v(x_{\h,j})
= k(x,Y_{\h})^T W_{\h} R_{\h} v \restrict{X_{\h}}
\quad \text{for all $v \in C^0(\Omegabar),$}
\end{equation}
with $k(x,Y_{\h})$ being the column vector of values
$k(x,y_{\h,j})$ for $j = 1,\dots,\abs{Y_{\h}}$.
By the usual assumptions on the smoothness of $k$,
\eqref{eq:Kh} also defines a family of operators
$\Kcal_{\h} \colon C^q(\Omegabar) \to C^q(\Omegabar)$.
Just like in the case of the classical Nyström scheme,
the linear system \eqref{eq:nystrom-decoupled}
can be solved if and only if
the semidiscrete operator equation
$(\lambda \Ical - \Kcal_{\h}) u_{\h} = f$
can be solved in $C^0(\Omegabar)$.

\begin{proposition} \label{prop:interp-decoupled}
For any $k \in C^0(\Omegabar \times \Omegabar)$
and $f \in C^0(\Omegabar)$,
the linear system $(\lambda I_{\h} - K_{\h} W_{\h} R_{\h})
\uhatvec_{\h} = f \restrict{X_{\h}}$ has solutions
if and only if the operator equation
$(\lambda \Ical - \Kcal_{\h}) u_{\h} = f$
has solutions in $C^0(\Omegabar)$. Moreover,
if the two solution sets are not empty,
they have the same dimension as affine spaces.
\end{proposition}

\begin{proof}
Let $u_{\h}$ be a continuous solution of the operator equation.
Then a solution of the linear system can be constructed
by evaluating $u_{\h}$ at the solution nodes $X_{\h}$:
\begin{gather*}
(\lambda \Ical - \Kcal_{\h}) u_{\h} = f
\quad
\lambda u_{\h}(x) - k(x,Y_{\h})^T W_{\h} R_{\h} u_{\h | X_{\h}} = f(x)
\quad
\lambda u_{\h | X_{\h}} - k(x,Y_{\h})^T_{|X_{\h}} W_{\h} R_{\h} u_{\h | X_{\h}} = f_{|X_{\h}} \\
\lambda u_{\h | X_{\h}} - K_{\h} W_{\h} R_{\h} u_{\h | X_{\h}} = f_{|X_{\h}}
\quad
(\lambda I_{\h} - K_{\h} W_{\h} R_{\h}) u_{\h | X_{\h}} = f \restrict{X_{\h}}.
\end{gather*}
Conversely, let $\uhatvec_{\h}$ be a solution of the
linear system, and define
\begin{equation} \label{eq:interp-decoupled}
u_{\h}(x) = \frac{1}{\lambda} \Bigl(
k(x,Y_{\h})^T W_{\h} R_{\h} \uhatvec_{\h} + f(x)
\Bigr).
\end{equation}
The regularity of $u_{\h}$ follows from the assumptions
on $k$ and $f$.
Evaluating $u_{\h}$ at $x_{\h,i}$ immediately proves
that $u_{\h}(x_{\h,i}) = \uhatvec_{\h,i}$.
The vector $\uhatvec_{\h}$ in \eqref{eq:interp-decoupled}
can hence be replaced by $u_{\h | X_{\h}}$,
which shows that $u_{\h}$ is a solution of the operator equation.

By the linearity of the system and of the operator equation,
the solution sets are affine spaces. They have the same
dimension because pointwise evaluation at
solution nodes $X_{\h}$ and \eqref{eq:interp-decoupled}
are both affine maps, and we have just showed that one
is the inverse of the other (when the domains of the
two maps are restricted to the solution sets).
\end{proof}

The expression \eqref{eq:interp-decoupled} generalizes
the \emph{Nyström interpolation formula}
\eqref{Nistrominterp} to our decoupled setting.
The importance of Proposition~\ref{prop:interp-decoupled}
is that the existence of solutions to the semidiscrete operator
equation $(\lambda \Ical - \Kcal_{\h}) u_{\h} = f$ can be
investigated using tools from functional analysis,
unlike the existence of solutions to the linear system
\eqref{eq:nystrom-decoupled}.
A powerful tool for the analysis of Nyström-like numerical
methods is the framework of collectively compact operator
approximations, see Section~4.1.2 of \cite{atkinson1997numerical}
for an introduction and the book \cite{anselone1971collectively}
for an extensive treatment of the subject.
For the reader's convenience, we recall here the key
definition of collectively compact family of operators.

\begin{definition}
Let $\Tcal$ be a family of bounded linear operators
over a Banach space $V$. We say that $\Tcal$
is \emph{collectively compact} if the set
\[
\bigl\{ Tv \mid T \in \Tcal, \, v \in V, \, \norm{v} \leq 1 \bigr\}
\]
is relatively compact, i.e.~has compact closure.
Equivalently, for every sequence
\[
(T_i,v_i) \subset \Tcal \times \{v \in V \mid \norm{v} \leq 1\},
\quad i \in \N,
\]
it must be possible to extract a subsequence $(T'_i,v'_i)$
such that $\{T'_i v'_i\}_{i \in \N}$ is convergent in $V$.
\end{definition}

Observe that, in the special case of a family $\Tcal = \{T\}$
consisting of a single operator $T$, this definition
is equivalent to compactness of $T$.

We can now prove that, under suitable assumptions
on the kernel, on the quadrature scheme, and on the
reconstruction scheme, the $\h$-indexed family
of linear operators $\Kcal_{\h} \colon C^0(\Omegabar) \to C^0(\Omegabar)$
is collectively compact. This property is stronger than individual compactness
of $\Kcal_{\h}$ for all $\h > 0$, and will play a key
role in our convergence analysis of the decoupled Nyström scheme.

\begin{lemma} \label{lemma:cc}
Let $k$ be a function in $C^0(\Omegabar \times \Omegabar)$,
let $\{(Y_{\h},\w_{\h})\}_{\h>0}$ be a stable quadrature scheme,
and let $\{(X_{\h},Y_{\h},R_{\h})\}_{\h>0}$ be a stable interpolation
scheme. Then, the $\h$-indexed family of linear operators
$\Kcal_{\h} \colon C^0(\Omegabar) \to C^0(\Omegabar)$
defined in \eqref{eq:Kh} is collectively compact.
\end{lemma}

\begin{proof}
We need to show the relative compactness of the set
\begin{equation} \label{eq:relcompactness}
\left\{ \Kcal_{\h} v \mid \h = (h_X,h_Y) > 0, \:
v \in C^0(\Omegabar) \text{ such that }
\norm{v}_{\infty} \leq 1 \right\}.
\end{equation}
Let $B$ denote the closed unit ball in $C^0(\Omegabar)$.
By the Arzelà-Ascoli theorem \ref{prop:asc-arz},
this amounts to checking two properties of the
elements of the set \eqref{eq:relcompactness}:
\begin{itemize}
\item Uniform boundedness, i.e.~there exists a constant
$C$ such that
\[
\norm{\Kcal_{\h} v}_\infty \leq C
\quad \text{for all $\h > 0$, $v \in B$}.
\]
\item Uniform equicontinuity, i.e.~for all $\varepsilon > 0$,
there exists $\delta > 0$ such that 
\[
\abs{ (\Kcal_{\h} v)(x)-(\Kcal_{\h} v)(x') } < \varepsilon
\quad \text{for all $\h > 0$, $v \in B$,
$x,x' \in \Omegabar$ such that $\norm{x-x'} < \delta$}.
\]
\end{itemize}

Let $C_Q$ and $C_I$ be the upper bounds on the stability
constants of the quadrature and interpolation schemes,
respectively, as in Definitions~\ref{CQ} and \ref{CI}.
For all $\h > 0$ and $v \in B$,
\begin{align} \label{eq:unifbound}
\norm{\Kcal_{\h} v}_\infty
&= \normBig{k(x,Y_{\h})^T W_{\h}
R_{\h} v \restrict{X_{\h}}}_{\infty,x \in \Omegabar}
= \normBigg{
    \sum_{i=1}^{\abs{Y_{\h}}}
    k(x,y_{\h,i}) \, w_{\h,i}
    \sum_{j=1}^{\abs{X_{\h}}} (R_{\h})_{ij}
    \, v(x_{\h,j})}_{\infty,x \in \Omegabar} \nonumber \\
&\leq \normBigg{
    \sum_{i=1}^{\abs{Y_{\h}}} \Big\lvert
    k(x,y_{\h,i}) \, w_{\h,i}
    \sum_{j=1}^{\abs{X_{\h}}} (R_{\h})_{ij}
    \, v(x_{\h,j}) \Big\rvert
    }_{\infty,x \in \Omegabar} \nonumber \\
&\leq \biggl( \sum_{i=1}^{\abs{Y_{\h}}}
    \abs{w_{\h,i}} \biggr) \normbigg{ \,
    \max_{i=1,\dots,|Y_{\h}|} \Bigl\lvert
    k(x,y_{\h,i})
    \sum_{j=1}^{\abs{X_{\h}}} (R_{\h})_{ij}
    \, v(x_{\h,j}) \Bigr\rvert
    }_{\infty,x \in \Omegabar} \\
&\leq C_Q \, \normbig{k}_{
    C^0(\Omegabar \times \Omegabar)} \,
    \normbig{R_{\h} \, v
    \restrict{X_{\h}}}_\infty \nonumber \\
&\leq C_Q \, \normbig{k}_{
    C^0(\Omegabar \times \Omegabar)} \,
    \normbig{R_{\h}}_\infty
    \normbig{v \restrict{X_{\h}}}_\infty
\leq C_Q \, C_I \, \normbig{k}_{
    C^0(\Omegabar \times \Omegabar)}. \nonumber
\end{align}

This chain of inequalities proves uniform boundedness.
Moving on to uniform equicontinuity, we begin by
observing that $k$ is uniformly continuous
on $\Omegabar \times \Omegabar$,
because it is continuous on a compact set.
This means that, for any $\hat{\varepsilon} > 0$,
there exists $\hat{\delta}(\hat{\varepsilon}) > 0$
such that
\[
\abs{k(x',y')-k(x,y)} < \hat{\varepsilon}
\quad \text{for all $x,x',y,y' \in \Omegabar$
such that $\norm{x'-x} + \norm{y'-y}
< \hat{\delta}(\hat{\varepsilon})$}.
\]
The notation $\hat{\delta}(\hat{\varepsilon})$
highlights the dependence of $\hat{\delta}$
on $\hat{\varepsilon}$. Fix $\varepsilon > 0$.
We seek $\delta > 0$ such that
\[
\abs{
  k(x, Y_{\h})^T W_{\h} R_{\h} v \restrict{X_{\h}}
- k(x',Y_{\h})^T W_{\h} R_{\h} v \restrict{X_{\h}}
} < \varepsilon
\]
holds for all $\h > 0$, for all $v \in B$,
and for all $x,x' \in \Omegabar$ such that
$\norm{x-x'} < \delta$. This is the case for
\[
\delta = \hat{\delta}(C_Q^{-1} C_I^{-1} \varepsilon),
\]
as shown by the following estimate (steps
in common with \eqref{eq:unifbound} are
omitted for brevity):
\begin{gather*}
\abs{ (\Kcal_{\h} v)(x) - (\Kcal_{\h} v)(x') }
\leq \abs{\Bigl( k(x,Y_{\h})^T -k(x',Y_{\h})^T \Bigr)
    W_{\h} R_{\h} v \restrict{X_{\h}}} \leq \\
\leq \max_{i=1,\dots,\abs{Y_{\h}}}
    \abs{k(x,y_{\h,i})^T - k(x',y_{\h,i})^T}
    \norm{w_{\h}}_1 \normbig{R_{\h}}_\infty
    \normbig{v \restrict{X_{\h}}}_\infty
< C_Q^{-1} C_I^{-1} \varepsilon \, C_Q C_I = \varepsilon. \qedhere
\end{gather*}
\end{proof}

We are now ready to state and prove the error analysis
of our decoupled Nyström scheme, which is the main theoretical
result in this paper.

\begin{theorem} \label{theo:nystrom-decoupled}
Consider the decoupled Nyström scheme \eqref{eq:nystrom-decoupled}
for the numerical solution of the Fredholm integral equation
of the second kind $(\lambda \Ical - \Kcal) u = f$,
and suppose that
\begin{enumerate}
\item The function $f$ belongs to $C^q(\Omegabar)$,
and the kernel $k$ defining $\Kcal$ belongs to $C^q(\Omegabar \times \Omegabar)$.
\item The scalar $\lambda$ does not belong to
the spectrum of $\Kcal$.
\item The quadrature scheme
$\{(Y_{\h},\w_{\h})\}_{\h>0}$ is stable and
has convergence order $q_w$, with $0 < q_w \leq q$.
\item The interpolation scheme
$\{(X_{\h},Y_{\h},R_{\h})\}_{\h>0}$
is stable and has convergence order $q_R$,
with $0 < q_R \leq q$.
\end{enumerate}
Then, for a sufficiently small pair $\h = (h_X,h_Y)$,
with $h_X,h_Y > 0$, the linear system
\begin{equation} \label{eq:nystrom-decoupled-theo}
A_{\h} \uhatvec_{\h}
= (\lambda I_{\h} - K_{\h} W_{\h} R_{\h}) \uhatvec_{\h}
= f \restrict{X_{\h}}
\end{equation}
has a unique solution, and the condition number of $A_{\h}$
in the $\infty$-norm is bounded independently of $\h$.
Moreover, denoting the Nyström interpolant~\eqref{eq:interp-decoupled}
by $u_{\h}$, there exists $C_s > 0$ such that
\begin{equation} \label{eq:errestdec}
\norm{u - u_{\h}}_{\infty}
\leq C_s \left(
C_w \norm{\norm{k(x,y)}_{C^{q_w}(\Omegabar), y}}_{\infty, x}
    \norm{u}_{C^{q_w}(\Omegabar)} h_Y^{q_w}
+ C_Q C_R \norm{k}_{C^0(\Omegabar \times \Omegabar)}
    \norm{u}_{C^{q_R}(\Omegabar)} h_X^{q_R} \right),
\end{equation}
with $C_w,C_Q,C_R$ being the constants in the definitions
of stable and convergent quadrature and reconstruction schemes,
see Definitions~\ref{Cw}, \ref{CQ}, \ref{CR}.
\end{theorem}

\begin{proof}
The main idea is to use the framework of collectively
compact operator approximations to show that the
$\h$-indexed operators
\begin{equation} \label{eq:likcal}
\lambda \Ical - \Kcal_{\h} \colon C^0(\Omegabar) \to C^0(\Omegabar),
\quad \h = (h_X,h_Y) > 0
\end{equation}
are invertible for sufficiently small $\h$,
and that the norm of the inverse operators can be bounded
uniformly with respect to $\h$. Once these properties have been
established, the connection between \eqref{eq:likcal}
and the linear system \eqref{eq:nystrom-decoupled-theo}
is given by Proposition~\ref{prop:interp-decoupled},
and the rest of the proof goes through easily.

By the Fredholm alternative and the assumptions on $\lambda$,
$k$, $f$, the operator
$\lambda \Ical - \Kcal \colon C^0(\Omegabar) \to C^0(\Omegabar)$
is invertible, and there exists a unique solution
$u \in C^q(\Omegabar)$ to the integral equation $(\lambda \Ical - \Kcal) u = f$.
The regularity lifting from $C^0(\Omegabar)$ to $C^q(\Omegabar)$
has already been discussed in Section~\ref{sec:fie}.
The algebraic identity
\[
\lambda \Ical - \Kcal_{\h}
= (\lambda \Ical - \Kcal)
\Bigl( \Ical + (\lambda\Ical - \Kcal)^{-1} (\Kcal-\Kcal_{\h}) \Bigr)
\]
suggests proving that $\norm{\Kcal-\Kcal_{\h}}_{C^0(\Omegabar)} \to 0$
for $\h \to 0$, so that
\[
\Ical + (\lambda\Ical - \Kcal)^{-1} (\Kcal-\Kcal_{\h}) \to \Ical,
\]
which would in turn imply that $\lambda \Ical - \Kcal_{\h}$ is the
composition of two invertible operators for sufficiently small $\h$.
This approach, which is routinely followed in the analysis of collocation
and Galerkin methods, cannot work for Nyström-like methods since one can prove that
\[
\lim_{\h \to 0^+} \norm{\Kcal-\Kcal_{\h}}_{C^0(\Omegabar)} > 0.
\]
Using the framework of collectively compact operators,
however, one can establish the weaker result
\begin{equation} \label{eq:weaker-conv-Kcal}
\lim_{\h \to 0^+} \norm{(\Kcal-\Kcal_{\h}) \Kcal_{\h}}_{C^0(\Omegabar)} = 0,
\end{equation}
which is nevertheless sufficient to prove that
$\lambda \Ical - \Kcal_{\h}$ is invertible
for sufficiently small $\h$ using functional analytic methods.
Lemma~4.1.2 in \cite{atkinson1997numerical} states that
\eqref{eq:weaker-conv-Kcal} holds under three assumptions:
\begin{description}
\item[P1.] $\Kcal$ and all elements of $\{\Kcal_{\h}\}_{\h > 0}$
are linear operators from $C^0(\Omegabar)$ to $C^0(\Omegabar)$.
\item[P2.] $\lim_{\h \to 0} \norm{\Kcal v - \Kcal_{\h} v}_{\infty} = 0$
for all $v \in C^0(\Omegabar)$.
\item[P3.] The family $\{\Kcal_{\h}\}_{\h > 0}$ is collectively compact.
\end{description}
The lemma is actually stated for sequence-indexed families
$\{T_i\}_{i \in \N}$ and not for a continuous two-parameter
index such as $\h = (h_X,h_Y)$, but the standard
sequential-criterion reduction (testing every sequence
$\h_i \to 0$) extends it to our setting without further change.
P1 holds by definition, and P3 follows from
Lemma~\ref{lemma:cc}. As for P2, we rely on a density argument.
The density of $C^q(\Omegabar)$ in $C^0(\Omegabar)$ is
proved by a classical mollifier argument, which we recall
in Lemma~\ref{lemma:density}, and the uniform boundedness of
$\{\Kcal_{\h}\}_{\h > 0}$ is part of
Lemma~\ref{lemma:cc}.
Let $C_r$ be a constant such that
\[
\norm{\Kcal_{\h}}_{C^0(\Omegabar)} \leq C_r
\quad \text{for all $\h > 0$.}
\]
By Proposition~\ref{prop:pointwise-conv-Cq}, assumption P2
only needs to be checked for all $v \in C^q(\Omegabar)$
instead of $v \in C^0(\Omegabar)$.
In this way, $v$ is sufficiently regular to ensure convergence
of the quadrature and reconstruction schemes. The following
estimate not only proves P2 for $v \in C^q(\Omegabar)$,
but will also lead to the inequality~\eqref{eq:errestdec},
and is therefore one of the most important steps in the whole proof.
By the assumptions on the order of convergence of the quadrature
and reconstruction schemes, and the submultiplicativity of
$\norm{\cdot}_{C^q(\Omegabar)}$ proved in Proposition~\ref{prop:submult},
\begin{align*}
\norm{\Kcal v - \Kcal_{\h} v}_\infty
&= \norm{\int_\Omega k(x,y) v(y) \dy
    - k(x,Y_{\h})^T W_{\h} R_{\h} v \restrict{X_{\h}}}_{\infty,x} \\
&\leq \normbigg{\int_\Omega k(x,y) v(y) \dy
    - k(x,Y_{\h})^T W_{\h} v \restrict{Y_{\h}}}_{\infty,x}
+ \normbigg{k(x,Y_{\h})^T W_{\h} (v \restrict{Y_{\h}}
    - R_{\h} v \restrict{X_{\h}})}_{\infty,x} \\
&\leq \norm{ C_w h_Y^{q_w}
    \norm{k(x,y)v(y)}_{C^{q_w}(\Omegabar), y}
    }_{\infty,x}
+ \normBig{
\sum_{i=1}^{\abs{Y_{\h}}}
    k(x,y_{\h,i}) \, w_{\h,i}
    \Bigl( v(y_{\h,i})-
    \sum_{j=1}^{\abs{X_{\h}}} (R_{\h})_{ij}
    \, v(x_{\h,j}) \Bigr)
}_{\infty,x} \\
&\leq C_w h_Y^{q_w} \norm{
    \norm{k(x,y)}_{C^{q_w}(\Omegabar), y}}_{\infty, x}
    \norm{v}_{C^{q_w}(\Omegabar)}
+ \norm{w_{\h}}_1 \norm{\norm{k(x,y)}_{\infty,y}}_{\infty,x}
 \norm{v \restrict{Y_{\h}} - R_{\h} v \restrict{X_{\h}}}_{\infty} \\
&\leq C_w h_Y^{q_w} \norm{
    \norm{k(x,y)}_{C^{q_w}(\Omegabar), y}}_{\infty, x}
    \norm{v}_{C^{q_w}(\Omegabar)}
+ C_Q \norm{k}_{C^0(\Omegabar \times \Omegabar)}
    C_R h_X^{q_R} \norm{v}_{C^{q_R}(\Omegabar)}.
\end{align*}
Since $q_w,q_R > 0$, it follows that
$\lim_{\h \to 0^+} \norm{ \Kcal v - \Kcal_{\h} v }_\infty = 0$
for all $v \in C^q(\Omegabar)$, as required.

Now that P1-P3 have been checked, Lemma~4.1.2 in
\cite{atkinson1997numerical} establishes the
limit~\eqref{eq:weaker-conv-Kcal}. In turn,
by Theorem~4.1.1 in the same book, any sufficiently small
$\h = (h_X,h_Y) > 0$ for which
\[
\norm{(\Kcal-\Kcal_{\h}) \Kcal_{\h}}_{C^0(\Omegabar)}
< \frac{\abs{\lambda}}{\norm{(\lambda \Ical - \Kcal)^{-1}}_{C^0(\Omegabar)}}
\]
ensures that $\lambda \Ical - \Kcal_{\h}$ is invertible,
and keep $\norm{(\lambda \Ical - \Kcal_{\h})^{-1}}_{C^0(\Omegabar)}$
uniformly bounded with respect to the pair of discretization
parameters $\h = (h_X,h_Y)$. Let $C_s$ be a constant such that
\[
\norm{(\lambda \Ical - \Kcal_{\h})^{-1}}_{C^0(\Omegabar)} \leq C_s
\quad \text{for sufficiently small $\h > 0$.}
\]
We have proved that, for sufficiently small $\h$,
the operator equation $(\lambda \Ical - \Kcal_{\h}) u_{\h} = f$
has a unique solution $u_{\h} \in C^0(\Omegabar)$ and so,
by Proposition~\ref{prop:interp-decoupled}, the linear
system~\eqref{eq:nystrom-decoupled-theo} also has a unique
solution $\uhatvec_{\h}$. Moreover, the Nyström interpolation
formula~\eqref{eq:interp-decoupled} constructed from $\uhatvec_{\h}$
must coincide with the solution $u_{\h}$,
and so $u_{\h} \in C^q(\Omegabar)$.
Another consequence of Theorem~4.1.1 in \cite{atkinson1997numerical}
is that
\begin{align*}
\norm{u-u_{\h}}_{\infty}
&\leq \norm{(\lambda \Ical - \Kcal_{\h})^{-1}}_{C^0(\Omegabar)}
    \norm{ \Kcal u - \Kcal_{\h} u }_\infty \\
&\leq C_s \Bigl( C_w h_Y^{q_w} \norm{
    \norm{k(x,y)}_{C^{q_w}(\Omegabar), y}}_{\infty, x}
    \norm{u}_{C^{q_w}(\Omegabar)}
    + C_Q \norm{k}_{C^0(\Omegabar \times \Omegabar)}
    C_R h_X^{q_R} \norm{u}_{C^{q_R}(\Omegabar)} \Bigr),
\end{align*}
which proves~\eqref{eq:errestdec}.
Finally, the analysis at the end of Section~4.1.1
of \cite{atkinson1997numerical} applies with minimal
changes to our setting and establishes that
\[
\kappa_{\infty}(A_{\h})
\deq \normbig{A_{\h}}_\infty \normbig{A_{\h}^{-1}}_\infty
\leq \norm{\lambda \Ical - \Kcal_{\h}}_{C^0(\Omegabar)}
    \norm{(\lambda \Ical - \Kcal_{\h})^{-1}}_{C^0(\Omegabar)}
\leq (\abs{\lambda}+C_r) C_s. \qedhere
\]
\end{proof}

The error estimate~\eqref{eq:errestdec} can be understood
as the sum of two parts: a \emph{quadrature term} that scales as
$h_Y^{q_w}$, and a \emph{reconstruction term} that scales as $h_X^{q_R}$.
Remarkably, the kernel-dependent constants in the two terms are
unbalanced: the reconstruction term is proportional to
$\norm{k}_{C^0(\Omegabar \times \Omegabar)}$,
whereas the quadrature term is proportional to the potentially
much larger value
\[
\norm{\norm{k(x,y)}_{C^{q_w}(\Omegabar), y}}_{\infty, x}
\]
which involves the $\infty$-norms of mixed partial derivatives
of $k$ up to the order $q_w$. Especially for narrow kernels
and moderately high convergence orders such as $q_w = 5$,
the ratio
\begin{equation} \label{eq:k-ratio}
\frac{\norm{\norm{k(x,y)}_{C^{q_w}(\Omegabar), y}}_{\infty, x}}
{\norm{k}_{C^0(\Omegabar \times \Omegabar)}}
\end{equation}
can reach several orders of magnitude in size. A rigorous analysis
of \eqref{eq:k-ratio} as a function of $q_w$ and $\sigma$
for Gaussian kernels \eqref{eq:gaussian} would provide valuable
insight, however we consider it to be beyond the scope
of the current paper.

Even though the error estimate~\eqref{eq:errestdec}
is not guaranteed to be tight, one can try to choose
free parameters in a decoupled Nyström scheme
(such as the decoupling ratio $\gamma = h_X/h_Y$)
according to a balancing principle:
the contribution to the overall error of
the quadrature term should match that of the reconstruction term.
In this way, time will not be wasted
computing quadrature formulas whose accuracy is bottlenecked
by the reconstruction of values at $Y_{\h}$, and, likewise,
time will not be wasted with reconstructions
whose accuracy is bottlenecked by numerical integration.

Then, to ensure balance,
the decoupling ratio $\gamma = h_X/h_Y$ should be
\begin{equation} \label{eq:gamma-star}
\gamma^*
\deq \left( \frac{C_w \norm{\norm{k(x,y)}_{C^{q}(\Omegabar), y}}_{\infty, x}}
{C_Q C_R \norm{k}_{C^0(\Omegabar \times \Omegabar)}}
\right)^{1/q}
\quad \text{assuming $q_w = q_R = q$.}
\end{equation}
Optimal constants $C_w,C_Q,C_R$ are typically unknown
a priori, and \eqref{eq:k-ratio} might be hard to estimate
for arbitrary $k$ and $\Omegabar$, so this choice of
$\gamma$ is purely of theoretical interest.
Crucially, however, $\gamma^*$ does not depend on the size
of the discretization parameters $h_X$ and $h_Y$.
This independence suggests a practical strategy:
find a nearly optimal value of $\gamma$ via trial and error
on a coarse set of nodes, and then use the same decoupling
ratio for computationally intensive simulations on finer sets.
This is the approach that will be followed in
Section~\ref{sec:numerical-experiments} of this paper,
devoted to numerical experiments.

\subsection{The decoupled Nyström method in a meshless setting}
\label{ssec:decoupled-meshless}
The decoupled Nyström method has been formulated and analyzed
for general sets of solution nodes $X_{\h}$ and quadrature nodes $Y_{\h}$,
with no assumptions on their distribution aside from the stability
and convergence of the associated reconstruction and quadrature schemes
\[
\bigl\{(X_{\h},Y_{\h},R_{\h})\bigr\}_{\h>0}
\quad \text{and} \quad
\bigl\{(Y_{\h},\w_{\h})\bigr\}_{\h>0}.
\]
A concrete choice of such schemes is required to assemble
system \eqref{eq:nystrom-decoupled} and get a complete numerical method.
To avoid imposing additional structure on $X_{\h}$ and $Y_{\h}$,
in this paper we have chosen to numerically evaluate the
performance of our decoupled Nyström method in a meshless setting:
$X_{\h}$ and $Y_{\h}$ are sets of scattered points, generated
in a nearly uniform way over $\Omega$ according to independent
discretization parameters $h_X,h_Y > 0$. We do not consider here
locally refined point clouds, even though our approach would be
flexible enough to handle e.g.~a posteriori adaptation of nodes.

Like most meshless methods, we make use of radial basis functions (RBF),
which are a fundamental tool in approximation theory, especially
when dealing with unstructured, scattered data.
Another powerful approximation method is moving least squares (MLS);
an introduction to both RBF and MLS is given in \cite{wendland2004scattered}.
The recent article \cite{mezzanotte24} reviews the existing
literature on RBF-based methods for the numerical solution
of Fredholm integral equations of the second kind.
Meshless methods for such equations have attracted attention from several
authors: in \cite{mirzaei2010meshless,assari2013meshless,fatahi2016new}
meshless collocation schemes are proposed with a solution
space made of, respectively, MLS basis functions and cardinal functions
from global/local RBF interpolation. The collocation integrals are however
evaluated using composite Gaussian quadrature, so the schemes can only be
considered to be partially meshless. Moreover, a Nyström scheme on the same quadrature
nodes would be just as accurate, and so the benefit of collocation with only
the solution space being meshless is, to us, unclear. In Section 3.2 of
\cite{assari2013meshless}, to overcome this limitation, collocation integrals
are evaluated using Monte Carlo quadrature, which however leads to a method
with very low order of convergence with respect to the number of quadrature nodes.

The article \cite{smith2018nearest} is, to the best of our knowledge,
the only instance in the literature where the author has modified
a Nyström scheme to separate solution nodes from quadrature nodes,
motivated by the observation that the kernel varies much more rapidly
than the unknown solution in his problem of interest. The numerical method
is fully meshless and achieves a significant increase in computational
efficiency over classical Nyström. It is, however, tightly linked
to the method of regularized stokeslets for the solution of Stokes
flow equations in the context of microscale biological fluid dynamics,
and so a decoupled Nyström scheme was not developed in its generality.
Moreover, because of the error introduced by approximating the fundamental
solution of the Stokes equation (a singular function) with a smooth kernel,
the author had no reason to consider quadrature and reconstruction schemes
of order greater than one. An error analysis of this first-order method
was carried out in \cite{gallagher2019sharp}.

A fully meshless, high-order classical Nyström method has recently been
developed using quadrature schemes that integrate exactly various kinds
of radial basis functions centered at solution nodes \cite{cavoretto2025rbf}.
For RBFs that depend on a shape parameter $\varepsilon$, an automatic choice
$\varepsilon^*$ was made via leave-one-out cross-validation.
Computing quadrature weights by imposing exactness of a quadrature
formula for a basis of functions is an approach known as \emph{moment fitting},
and requires the solution of a linear system.
The main downsides of RBF moment fitting are the density of the system
matrix and the need to accurately compute the integrals of all radial
basis functions on $\Omega$ to assemble the right-hand side.
These challenges were overcome in \cite{davydov2025meshless},
where meshless moment-free quadrature formulas that scale up to hundreds
of thousands of nodes are introduced. By adopting in this paper the new
moment-free approach, which is independent of the shape of $\Omega$,
we could easily test our decoupled Nyström method on the nontrivial
domains shown in Figure~\ref{fig:test-domains}. Our choice of
meshless quadrature and reconstruction schemes is explained in
Section~\ref{ssec:meshless-quadrature-reconstruction}.

\subsubsection*{On the generation of scattered solution and quadrature nodes}

We end this section by describing the construction of scattered sets
of nodes $X_{\h}$ and $Y_{\h}$, since our computational domains
are not rectangles, where Cartesian grids or e.g.~Halton
points \cite{halton1960efficiency} can be readily defined.
For simplicity, we limit our exposition to the set $Y_{\h}$
with discretization parameter $h_Y$, since $X_{\h}$ is always
obtained in the same way.

Starting from a parametric description
of $\Omega$, which we assume to have piecewise smooth boundary,
we generate a set of boundary nodes $Z_{\h} \subset \partial\Omega$
that are approximately uniformly spaced with respect to arc length,
using $h_Y$ as the target distance between consecutive points.
For each node $z_i$ in $Z_{\h}$, a corresponding outward-pointing
normal $\nu_i$ is also computed. Then, an advancing front method
is used to extend the set $Z_{\h}$ into $Y_{\h}$ by
placing new nodes in the interior of $\Omega$. A queue of nodes $Q_{\h}$
is initialized as $Q_{\h} = Z_{\h}$, and then elements $q \in Q_{\h}$
are iteratively extracted and removed from the queue. Each $q$
generates 15 new candidate points for the advancement of the front,
uniformly distributed in a circle of radius $h_Y$ around $q$.
If a candidate point is outside of $\Omega$ or it is too close
to any of the points in the set $Y_{\h}$ constructed so far (i.e.~if
the minimum distance is smaller than $h_Y$), it is \emph{rejected}.
Otherwise, it is \emph{accepted} and added to the queue $Q_{\h}$
and the set $Y_{\h}$. The algorithm terminates when the queue is empty.
The resulting set $Y_{\h}$ covers $\Omegabar$ in a rather homogeneous way,
as shown in Figures~\ref{fig:test-domains-a} and \ref{fig:test-domains-b}.
More details on our meshless advancing front algorithm can be found
in Section 3.2 of the doctoral thesis \cite{degliesposti2025domain}.
An open source implementation in C++ with MATLAB bindings is available at
\[
\text{\url{https://github.com/BrunoDegliEsposti/NodeGenLib}}
\]
The performance of advancing front node generation is essentially limited
by the checks required to accept or reject each candidate point.
Minimum distance queries are answered efficiently using k-d trees,
and the fast and simple algorithm INNK-relocate introduced
in \cite{degliesposti2025domain} is used to check for inclusion in $\Omega$.
The algorithm does not take the parametrization of $\partial\Omega$ as input,
but rather the set of boundary nodes $Z_{\h}$ and the corresponding normals.
Inclusion can only be checked approximately, but the accuracy is sufficient
for advancing front methods, and the savings in computation time compared
to alternative methods are substantial.

As an alternative to advancing front methods, we consider rejection
sampling of Halton points to test the robustness of our decoupled Nyström
method in the case of less regular node distributions. Consider a bounding
box $\Omega_{\text{BB}}$ around $\Omega$, and let
\[
N_{\text{BB},Y} = \text{round}\bigl(1.1 \abs{\Omega_{\text{BB}}} h_Y^{-2}\bigr).
\]
The quantity $N_{\text{BB},Y}$ measures (up to the constant 1.1) the number of
$h_Y \times h_Y$ squares that fit inside the rectangle $\Omega_{\text{BB}}$.
In this context, the discretization parameter $h_Y$ is known
as \emph{packing spacing}. Since $h_Y$ now has a different meaning
compared to the advancing front algorithm, the constant 1.1 is introduced
to roughly equalize the number of nodes generated by the two methods
for the same $h_Y$. Note that the constant essentially depends on the
domain dimension and not on its shape.
Let $Y_{\text{BB},\h}$ be the set of $N_{\text{BB},Y}$ consecutive
elements of the Halton sequence (see \cite{halton1960efficiency}
and MATLAB's \texttt{haltonset} function), chosen
with a random starting index, and affinely mapped from $[0,1]^2$
to $\Omega_{\text{BB}}$. Then, we define $Y_{\h}$ as the intersection
of $Y_{\text{BB},\h}$ and $\Omega$. Inclusion in the domain is checked
using algorithm INNK-relocate, which takes as input the set $Z_{\h}$
and corresponding normals generated in the advancing front case.
To make the comparison more fair between the two node generation methods,
we add $Z_{\h}$ to the set $Y_{\h}$ generated by rejection sampling
of Halton points. An example of the result is shown
in Figure~\ref{fig:test-domains-c}.

\begin{figure}[tbp]
\centering
\hspace*{\fill}
\subcaptionbox{Cassini oval $\Omega_1$, advancing
front node generation.
\label{fig:test-domains-a}}{
\includegraphics[width=0.25\textwidth]{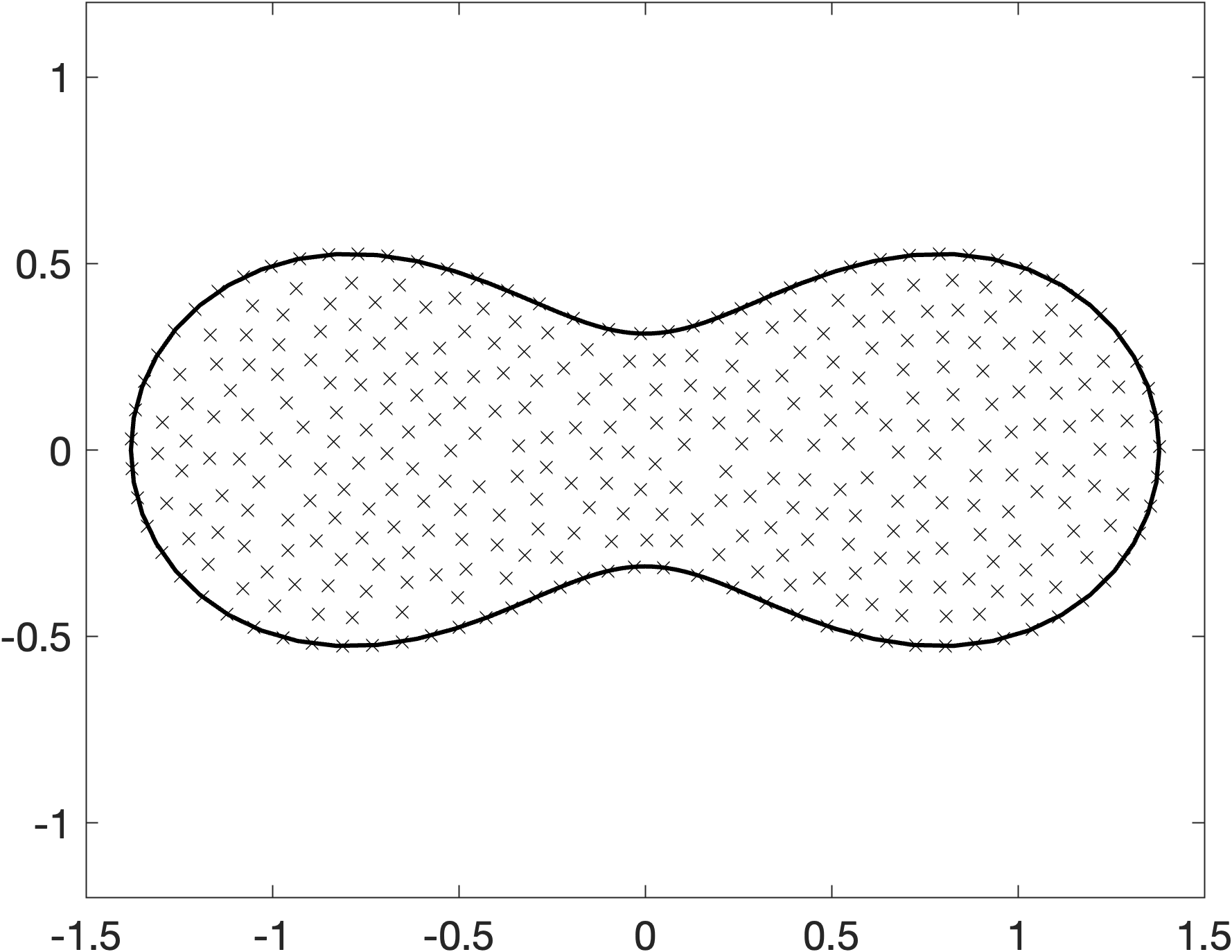}
} \hspace*{\fill}
\subcaptionbox{Domain $\Omega_2$, advancing
front node generation.
\label{fig:test-domains-b}}{
\includegraphics[width=0.25\textwidth]{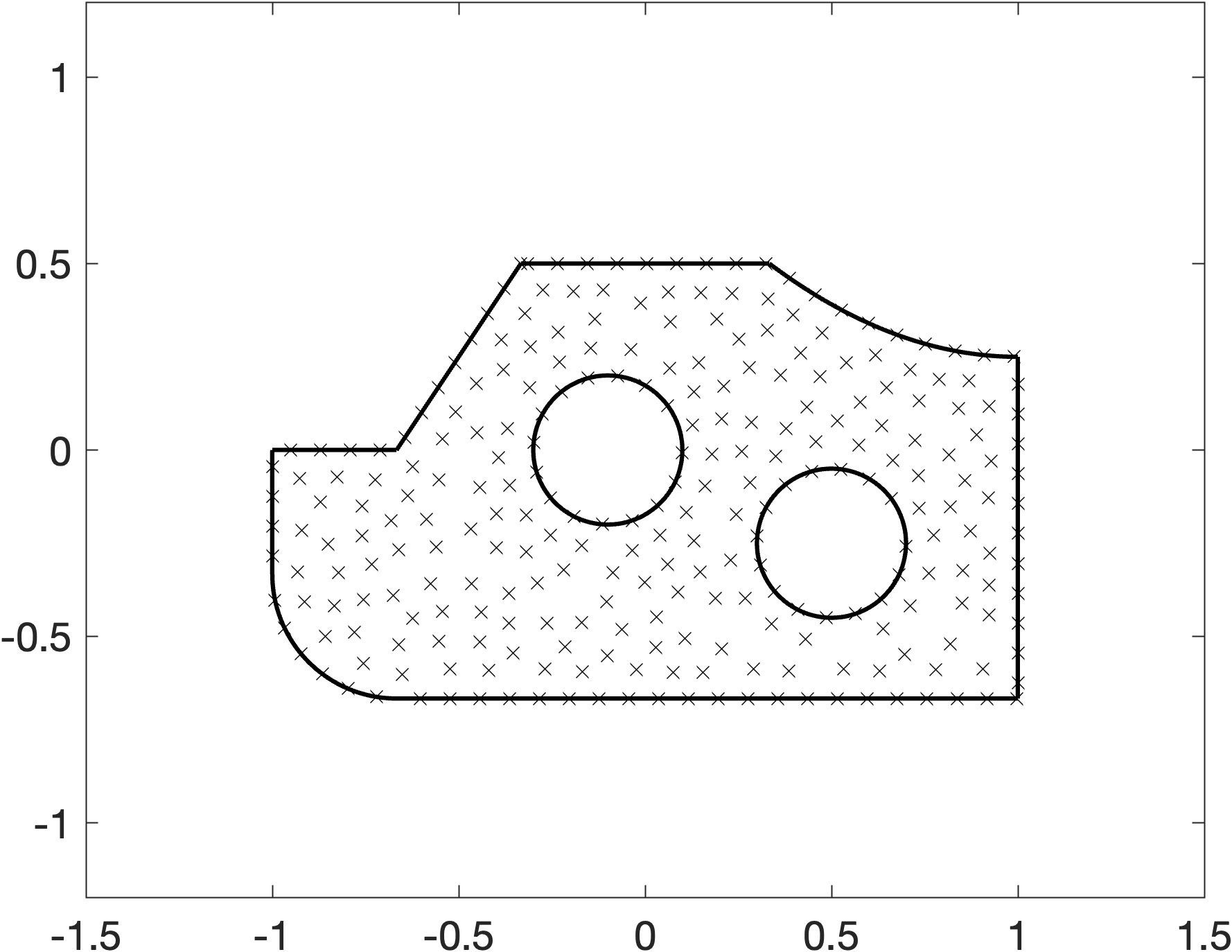}
} \hspace*{\fill}
\subcaptionbox{Domain $\Omega_2$, Halton points
rejection sampling.
\label{fig:test-domains-c}}{
\includegraphics[width=0.25\textwidth]{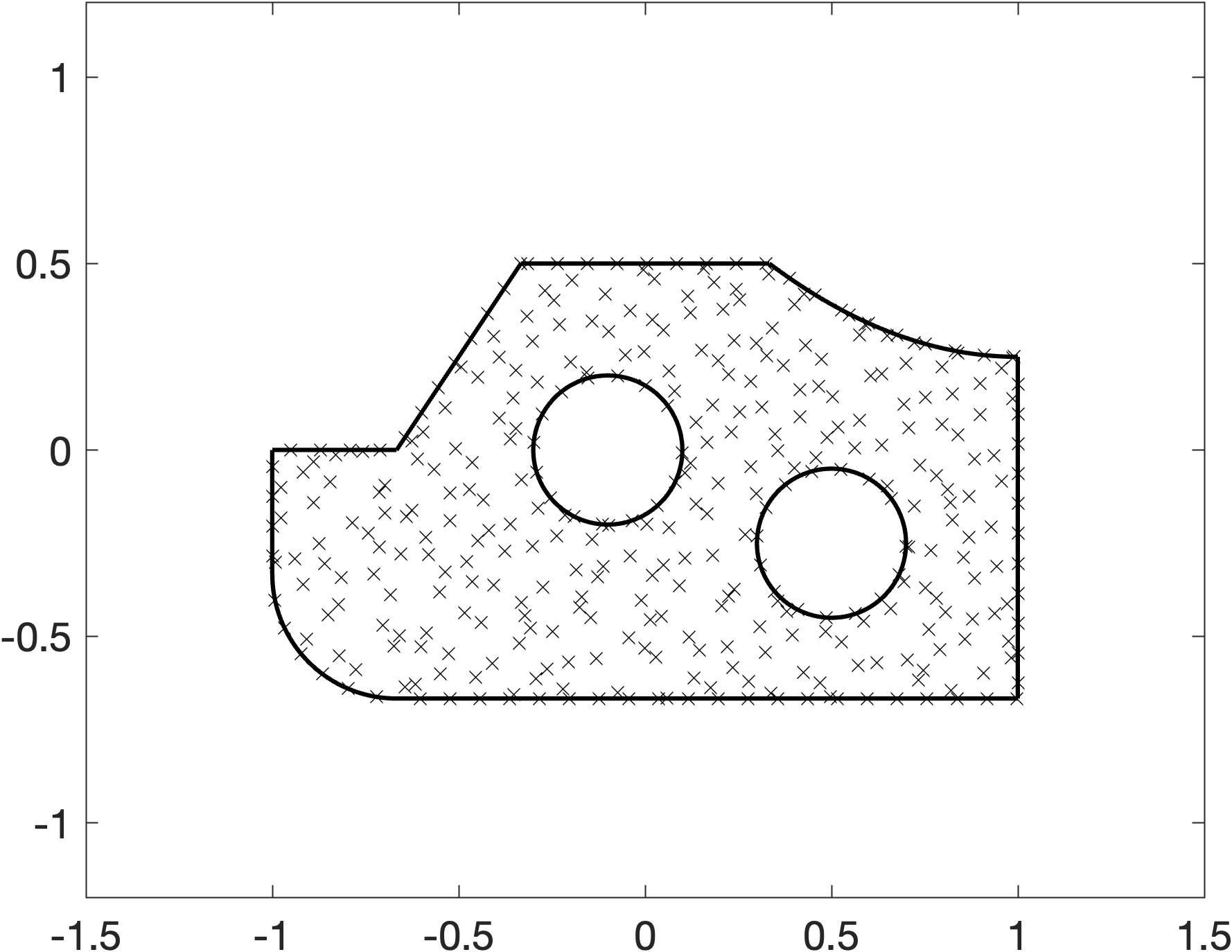}
} \hspace*{\fill}
\caption{Domains used in numerical tests.
A set of quadrature nodes with $h_Y = 0.08$
is displayed as crosses on top of the 2D domains.
Note that, by construction, some nodes are
placed on the boundary.}
\label{fig:test-domains}
\end{figure}

\subsection{On the choice of meshless quadrature and reconstruction schemes}
\label{ssec:meshless-quadrature-reconstruction}
The convergence of our decoupled Nyström scheme has been
proved in Theorem~\ref{theo:nystrom-decoupled} under the assumption
that the reconstruction scheme is stable, a key hypothesis
when establishing collective compactness of operators $\Kcal_{\h}$.
Even though it might be possible in practice to relax the bound
\[
\norm{R_{\h}}_{\infty} \leq C_I
\quad \text{for all $h_X > 0$ and $h_Y > 0$}
\]
with a sufficiently mild growth condition on $\norm{R_{\h}}_{\infty}$
as $h_X \to 0$, we have chosen a stable method for our numerical experiments.
We have defined reconstructions to be a slightly more general concept
than interpolation or quasi-interpolation: for any function $v$,
instead of seeking a global approximant $\tilde{v}$, we only aim
to approximate values of $v$ at $Y_{\h}$ by linearly combining
values of $v$ at $X_{\h}$ using the weights in~$R_{\h}$.
Note that, whenever a global approximant $\tilde{v}$ is available,
$R_{\h}$ is constructed by expressing $\tilde{v}$ as a function
of $v_{|X_{\h}}$, and evaluating $\tilde{v}$ at $Y_{\h}$.
For this reason, the Lebesgue constant of an interpolation or
quasi-interpolation scheme is an upper bound on $\norm{R_{\h}}_{\infty}$.

Tools such as local polynomial reproduction and moving least squares
are classical examples of \emph{high-order, stable, meshless} reconstruction
schemes~\cite{wendland2004scattered}. More recently, it was shown
in~\cite{davydov2018minimal} that local polynomial reproduction obtained
by minimizing suitable absolute seminorms of the reconstruction weights
deliver particularly good accuracy in numerical experiments
and are empirically stable. In general, one can expect stability
from any reconstruction scheme based on oversampled local approximation
with scalable stencils of fixed size, see~\cite{davydov2019optimal}
for the crucial definition of \emph{scalability} and the construction
of near-optimal stencils in Sobolev spaces.
The observation that local RBF interpolation with polyharmonic splines (PHS),
which are free from the choice of a shape parameter, deliver scalable
stencils goes back to~\cite{iske2003approximation},
where the convergence order for functions in $C^q(\Omegabar)$
is also proved. As will be explained shortly, the meshless quadrature scheme
introduced in~\cite{davydov2025meshless} and used in our numerical
experiments (in its MFD variant) depends on numerical differentiation
formulas constructed via polyharmonic kernels. To simplify our code,
we have therefore used local PHS interpolation as our reconstruction scheme.
The numerical results in Figure~\ref{fig:test-choice-qR}
anyway suggest that the choice of the reconstruction scheme is not critical,
and that our results can be reproduced with any other stable, high-order method.
A posteriori checks on the stability of our reconstructions are
provided in Tables~\ref{tab:stability-advfront} and~\ref{tab:stability-halton}.

For completeness, we briefly recall here how local PHS approximation
can be used to reconstruct $v(y)$ from the values of $v$ at the $N$ nodes
in $X_{\h}$ closest to $y \in \Omega$, and refer the reader to
\cite{fornberg2015primer} for the missing details.
More generally, we aim to reconstruct $\Lcal v(y)$ with $\Lcal$
being a differential operator of order $\ell \geq 0$.
Fix $m \in \N, m \geq 1$ as the \emph{polynomial augmentation order}
of the method, a quantity that directly controls its convergence order,
and let
\[
N = 2M,
\qquad M = \binom{m-1+d}{m-1},
\]
i.e.~$N$ is two times the dimension $M$ of the space
$\Pi_m^d$ of $d$-dimensional polynomials of degree up
to $m-1$ (a standard choice since \cite{bayona2017role}).
The goal is to determine for any $y \in \Omega$ local
weights $a_1,\dots,a_N$ such that
\begin{equation} \label{eq:differentiation-formula}
\Lcal v(y) \approx \sum_{i=1}^N a_i v(x_i),
\end{equation}
with $X_{\h}(y) = \{x_1,\dots,x_N\}$ being the set of $N$ elements
in $X_{\h}$ closest to $y$. Consider the odd-power polyharmonic spline RBF
\[
\varphi(r) = r^{2m-1},
\]
expressed in the radial coordinate $r$, and let $p_1,\dots,p_M$
be the standard monomial basis of $\Pi_m^d$.
As a preconditioning step, we normalize the set $X_{\h}(y)$ by
subtracting $y$ and scaling its elements by $h_X^{-1}$:
\[
x_1' = \frac{x_1-y}{h_X},
\qquad \dots, \qquad
x_N' = \frac{x_N-y}{h_X}.
\]
The weights in \eqref{eq:differentiation-formula} are found
by extracting the first $N$ components of the
solution of the linear system
\[
\begin{pmatrix}
\varphi\bigl(\norm{x_1'-x_1'}\bigr) & \dots & \varphi\bigl(\norm{x_1'-x_N'}\bigr)
    & p_1(x_1') & \dots & p_M(x_1') \\
\vdots & & \vdots & \vdots & & \vdots \\
\varphi\bigl(\norm{x_N'-x_1'}\bigr) & \dots & \varphi\bigl(\norm{x_N'-x_N'}\bigr)
    & p_1(x_N') & \dots & p_M(x_N') \\
p_1(x_1') & \dots & p_1(x_N') & 0 & \dots & 0 \\
\vdots & & \vdots & \vdots & & \vdots \\
p_M(x_1') & \dots & p_M(x_N') & 0 & \dots & 0
\end{pmatrix}
\begin{pmatrix}
a_1' \\
\vdots \\
a_N' \\
b_1' \\
\vdots \\
b_M'
\end{pmatrix}
=
\begin{pmatrix}
(\Lcal \varphi)\bigl(\norm{x_1'}\bigr) \\
\vdots \\
(\Lcal \varphi)\bigl(\norm{x_N'}\bigr) \\
(\Lcal p_1)(0) \\
\vdots \\
(\Lcal p_M)(0)
\end{pmatrix}
\]
and scaling each $a_i'$ by $h_X^{-\ell}$ to get $a_i$. The linear system
is consistent if $X_{\h}(y)$ is unisolvent for $\Pi_m^d$,
and its condition number is scale-invariant.
As reported in \cite{flyer2016role}, polyharmonic kernels with an
even power of $r$ and a logarithmic term would give no significant differences
in the accuracy and stability of the numerical differentiation formula
\eqref{eq:differentiation-formula}, even though theoretical considerations
would suggest using these RBFs instead of $r^{2m-1}$ when $d$ is even.

The convergence order of \eqref{eq:differentiation-formula}
for $v \in C^q(\Omegabar)$ is the minimum between $q - \ell$ and $m - \ell$.
For this reason, in the numerical experiments of
Section~\ref{sec:numerical-experiments}, where the reconstruction weights
are computed by choosing $\Lcal$ as the identity operator, we set $m = q_R$.
Each row of $R_{\h}$ contains the $N$ weights of local PHS interpolation
corresponding to a different $y \in Y_{\h}$, and so $R_{\h}$
is a sparse matrix with only $N \abs{Y_{\h}}$ nonzero elements.
This in turn allows fast assembly of the (dense) system matrix
$A_{\h} = \lambda I_{\h} - K_{\h} W_{\h} R_{\h}$ in our decoupled Nyström method.
Note that, even though the reconstruction given by local PHS interpolation
is not a continuous function of $y$ (because of the way
the set of $N$ nearest nodes in $X_{\h}$ changes with $y$),
this has no effect on the regularity of the Nyström interpolant,
which only depends on the smoothness of $k$ and $f$.

\subsubsection*{Meshless moment-free quadrature formulas arising from numerical differentiation}
The construction of stable, high-order quadrature formulas on scattered sets
of nodes is a challenging but essential step in defining meshless Nyström schemes.
For this work, we have chosen to rely on the moment-free quadrature
formulas recently introduced in \cite{davydov2025meshless} for two reasons.
First, as explained in Section~\ref{ssec:decoupled-meshless},
they can easily handle domains with complex shapes by overcoming
the moment computation problem.
Second, meshless moment-free formulas can be computed for very large numbers of nodes;
quadrature sets $Y_{\h}$ with tens of thousands of nodes may arise when
considering large decoupling ratios $\gamma$, but cannot be handled
by some other methods such as direct RBF moment fitting.

A peculiar feature of the moment-free approach is that quadrature
weights $\w_{\h}$ are computed as part of a combined quadrature
designed to approximate the difference of integrals over~$\Omega$
and its boundary.
Let $Y_{\h} \subset \Omegabar$ and $Z_{\h} \subset \partial\Omega$
be nearly uniform sets of scattered nodes, such as the ones produced
by the advancing front method of Section~\ref{ssec:decoupled-meshless}
with discretization parameter $h_Y > 0$.
Given sufficiently smooth functions $v \colon \Omega \to \R$
and $g \colon \partial\Omega \to \R$, we seek quadrature schemes
$(Y_{\h},\w_{\h})$ and $(Z_{\h},\vec{\mu}_{\h})$ such that
\[
\int_{\Omega} v(x) \dx - \int_{\partial\Omega} g(x) \dsigma(x)
\; \approx \;
\w_{\h}^T \, v_{|Y_{\h}} - \vec{\mu}_{\h}^T \, g_{|Z_{\h}}.
\]
The reason for this combined approach is that $\w_{\h}$ and $\vec{\mu}_{\h}$
are simultaneously determined by solving a sparse linear system that
enforces the weights to satisfy a discrete counterpart of the divergence theorem
\[
\int_{\Omega} \diver(F)(x) \dx
- \int_{\partial\Omega} F(x) \cdot \nu(x) \dsigma(x)
= 0
\qquad \text{for any sufficiently smooth $F \colon \Omega \to \R^d$}.
\]
More precisely, the weights $\w_{\h}$ and $\vec{\mu}_{\h}$ are found as the
minimum-norm solution of an underdetermined system whose construction
depends on the specific kind of moment-free approach employed.
Out of the two described in \cite{davydov2025meshless}, namely
Algorithm~2.8 (MFD) and Algorithm~2.9 (BSP), we opted for the MFD
approach based on meshless finite difference formulas
\eqref{eq:differentiation-formula}, which we recall here.

Let $\Xi_{\h} \subset \Omegabar$ be a set of nodes generated by the
advancing front method with discretization parameter $h_{\Xi} = 1.6 \, h_Y$,
so that $\Xi_{\h}$ is slightly coarser than $Y_{\h}$.
Let $L_1,L_2 \in \R^{\abs{Y_{\h}} \times \abs{\Xi_{\h}}}$
be the sparse matrices whose rows contain the weights of numerical
differentiation formulas \eqref{eq:differentiation-formula} for the
partial derivatives $\partial_1, \partial_2$ at nodes in $Y_{\h}$
using pointwise values at $\Xi_{\h}$ and polynomial augmentation orders $q_w+1$:
\[
(\partial_1 v)_{|Y_{\h}} \approx L_1 \, v_{|\Xi_{\h}},
\qquad
(\partial_2 v)_{|Y_{\h}} \approx L_2 \, v_{|\Xi_{\h}}
\]
Similarly, let $\tilde{B} \in \R^{\abs{Z_{\h}} \times \abs{\Xi_{\h}}}$
be the sparse matrix containing the weights of reconstruction
formulas for boundary nodes in $Z_{\h}$ using pointwise values at $\Xi_{\h}$
and polynomial augmentation order $q_w$:
\[
v_{|Z_{\h}} \approx \tilde{B} \, v_{|\Xi_{\h}}
\]
Assemble now the diagonal matrices
$D_1,D_2 \in \R^{\abs{Z_{\h}} \times \abs{Z_{\h}}}$ whose elements
are the first and second coordinates of the outward-pointing normals
$\nu_{\h}$ at the set of boundary nodes $Z_{\h}$,
and concatenate horizontally $L_1,L_2$ into $L$,
and the products $D_1 \tilde{B}, D_2 \tilde{B}$ into $B$.
Let $\mathbbm{1}$ be the all-ones column vector of size $\abs{Z_{\h}}$.
The minimum-norm solution (in the 2-norm) of the underdetermined sparse linear system
\begin{equation} \label{eq:quadrature-system}
\begin{pmatrix}
L^T & -B^T \\
0 & \mathbbm{1}^T
\end{pmatrix}
\begin{pmatrix}
\w_{\h} \\
\vec{\mu}_{\h}
\end{pmatrix}
=
\begin{pmatrix}
0 \\
\abs{\partial\Omega}
\end{pmatrix}
\end{equation}
contains the weights of the combined quadrature.
Since the boundary weights $\vec{\mu}_{\h}$ are not used in our Nyström method,
they can be safely discarded after solving~\eqref{eq:quadrature-system},
even though the inclusion of boundary weights in~\eqref{eq:quadrature-system}
is essential in defining a moment-free method: note how the right-hand side
is homogeneous, except for the last entry enforcing exactness of $\vec{\mu}_{\h}$
for constant integrands.

It is proved in \cite{davydov2025meshless} that, assuming stability,
the theoretical convergence order of the quadrature scheme
$(Y_{\h},\w_{\h})$ is~$q_w$ for sufficiently regular integrands.
In the same paper, the convergence rate in numerical
experiments almost always exceeded $q_w$.
Since Nyström schemes inherit convergence rates from quadrature schemes
(assuming no other limiting factors), our decoupled meshless Nyström scheme
also demonstrates superconvergence with respect to~$q_w$,
see for example the error curves in Figure~\ref{fig:test-choice-qW}.
A posteriori checks on the stability of our moment-free quadrature are
provided in Tables~\ref{tab:stability-advfront} and~\ref{tab:stability-halton}.

%% file: section4.tex
\graphicspath{{./figures/}}

\pgfplotstableset{
	search path = ./tables,
	col sep = comma,
	string replace = {NaN}{},
	empty cells with = {-},
	set thousands separator = {},
	every column/.append style = %
		{zerofill, sci e, precision = 2}
}

\pgfplotstableset{
	columns/hX/.style={
		column name=$h_X$, sci
	},
	columns/hY/.style={
		column name=$h_Y$, sci
	},
	columns/NX/.style={
		column name=$\abs{X_{\h}}$, int detect
	},
	columns/NY/.style={
		column name=$\abs{Y_{\h}}$, int detect
	},
    columns/condA/.style={
        column name=$\kappa(A_{\h})$
    },
    columns/stabw/.style={
        column name=$\norm{w_{\h}}_1 \abs{\Omega_2}^{-1}$
    },
    columns/stabR/.style={
        column name=$\norm{R_{\h}}_{\infty}$
    },
    columns/errLinf/.style={
        column name=$e_{\infty}$
    },
    columns/errRMS/.style={
        column name=$e_{\text{RMS}}$
    },
    columns/errL2/.style={
        column name=$e_2$
    }
}

\section{Numerical experiments}
\label{sec:numerical-experiments}
In this section we numerically evaluate the
accuracy and efficiency of our decoupled
Nyström method based on meshless moment-free
quadrature formulas.
Nonhomogeneous Fredholm integral equations
of the second kind with smooth kernels
are solved on the 2D domains shown in
Figure~\ref{fig:test-domains}, often relying on
the method of manufactured solutions
to provide an exact reference solution
to evaluate numerical errors.
However, constructing the right-hand side $f$
from a given solution $u$ in the case of an
integral equation is not as easy as in the case
of a partial differential equation, where symbolic
differentiation rules can be applied to
the analytic expression of $u$.
In fact, looking at \eqref{eq:fie}, the integral
\begin{equation} \label{eq:manufsol}
(\Kcal u)(x_{\h,i}) = \int_\Omega k(x_{\h,i},y) u(y) \dy
\end{equation}
needs to be computed as accurately as possible for
all $x_{\h,i} \in X_{\h}$, a much more difficult task compared
to the evaluation of a differential operator applied to $u$.
A natural choice here would be to approximate
the integral in \eqref{eq:manufsol}
using the quadrature formula $(Y_{\h},\w_{\h})$,
or possibly an even finer formula
for increased accuracy.
However, this approach would introduce
an error in the evaluation of manufactured
$f$ that cannot
be separated from the discretization
error of the Nyström method: both errors
would affect the distance between
$u_{|X_{\h}}$ and the numerical solution $\uhatvec_{\h}$.
To avoid this problem and increase the reliability
of the method of manufactured solutions,
we evaluate the integral in \eqref{eq:manufsol}
to machine precision using a combination
of the open-source package
Chebfun~\cite{driscoll2014chebfun},
the divergence theorem,
and one-dimensional adaptive Gaussian quadrature
on each smooth piece of $\partial\Omega$.
More precisely, we start by approximating
the integrands in \eqref{eq:manufsol}
by a polynomial $p_{\h,i}(y)$ for each
$x_{\h,i} \in X_{\h}$.
Assuming that all integrands are infinitely
smooth, this task can be carried out with
spectral accuracy over a bounding box of $\Omega,$ using the \texttt{chebfun2}
routine provided by the open-source Matlab 
package Chebfun. Such a routine adaptively
chooses the polynomial bi-degree to ensure
that the smooth functions at hand are approximated
to machine precision on the whole bounding box.
To improve efficiency, Chebfun iteratively constructs
the polynomial in low-rank form, i.e.~as
the sum of products of univariate polynomials
with as few summands as possible, see
\cite{townsend2013extension} for more details.
Then, the vector field
\[
P_{\h,i}(y)
= P_{\h,i}(y_1,y_2)
= \begin{pmatrix}
\int_{0}^{y_1} p_{\h,i}(\eta,y_2) \deta \\
0
\end{pmatrix},
\]
which can be easily computed in closed form
because $p_{\h,i}$ is a polynomial, satisfies
$\diver(P_{\h,i}) = p_{\h,i}$, and so,
by the divergence theorem,
\[
\int_\Omega p_{\h,i}(y) \dy
= \int_{\partial\Omega} P_{\h,i}(y) \cdot \nu(y) \dsigma(y).
\]
Since an exact parametric description of
the boundary is available for all
test domains, the boundary integral can
be evaluated to machine precision by MATLAB's
one-dimensional adaptive Gaussian quadrature
routine \texttt{integral} with absolute
and relative tolerances set to $4\cdot10^{-16}$.
If the boundary is only piecewise smooth,
the contributions from each piece are computed
separately and summed.

The manufactured solution $u$ is a scaled and
translated version of the Franke function
\cite{franke1979critical}, as provided
by the \texttt{franke} command in MATLAB.
The Franke function is meant to be evaluated
over $[0,1]^2$, but the domains in
Figure~\ref{fig:test-domains}
are centered at the origin. For this reason,
we compose the Franke function with the bijective affine mapping
\[
(x_1,x_2) \mapsto
\left( \frac{x_1+1}{2}, \frac{x_2+1}{2} \right)
\]
from $[-1,1]^2$ to $[0,1]^2$.

The accuracy of the Nyström method can be assessed
in several ways. As explained in
Section~\ref{sec:numerical-scheme},
a remarkable property of the method
is that signed pointwise errors are not just
elements of some vector in $\R^{\abs{X_{\h}}}$,
but rather point samples of a smooth function
defined over $\Omegabar$, namely the
difference between the exact solution $u$ and
the Nyström interpolant $u_{\h}$:
\[
u_{|X_{\h}} - \uhatvec_{\h}
= (u-u_{\h})_{|X_{\h}},
\quad u-u_{\h} \in C^q(\Omegabar).
\]
This means that the infinity norm $\norm{u-u_{\h}}_\infty$
that appears in estimate \eqref{eq:errestdec}
can be approximated by sampling over any
validation set $V_{\h}$, not just $X_{\h}$.
For simplicity, we have chosen $V_{\h} = X_{\h}$
in all of our numerical tests, even though a finer
validation set would lead to a more accurate approximation
of $\norm{u-u_{\h}}_\infty$. We have
computed relative errors instead of
absolute errors to make results comparable
across different domains and parameters:
\[
e_{\infty}
= \frac{\norm{(u-u_{\h})_{|X_{\h}}}_{\infty}}
{\norm{u_{|X_{\h}}}_{\infty}}.
\]
Discrete quadratic norms of the error,
being less sensitive to local peaks in the error,
are also of practical interest.
In the literature on meshless methods,
under the assumption of quasi-uniform
nodes, the most common error metric is
the relative root mean square (RMS) error
\[
e_{\text{RMS}} = \sqrt{\frac{
\sum_{i=1}^{\abs{X_{\h}}} (u(x_{\h,i})-\uhatvec_i)^2
}{ \sum_{i=1}^{\abs{X_{\h}}} u(x_{\h,i})^2}}.
\]
However, a more natural error metric
in our setting is
\[
e_{2} = \sqrt{\frac{
\sum_{i=1}^{\abs{Y_{\h}}}
w_{\h,i} (u(y_{\h,i})-u_{\h}(y_{\h,i}))^2
}{ \sum_{i=1}^{\abs{Y_{\h}}} w_{\h,i} u(y_{\h,i})^2}},
\]
which approximates the continuous
$L^2$ norm of the error using
the quadrature formula $(Y_{\h},\w_{\h})$.
The RMS case instead corresponds to a quadrature
formula with set of nodes $X_{\h}$ and constant
weights, as one would obtain with lower-order
Monte Carlo or quasi-Monte Carlo numerical
integration. The values $u_{\h}(y_{\h,i})$
are given by the decoupled Nyström interpolation
formula \eqref{eq:interp-decoupled}.

\subsection{Validation of convergence order}
\label{ssec:validation-convergence-order}
Our first goal is to show that the convergence
order of a classical 2D Nyström scheme
based on the meshless moment-free quadrature
formulas of MFD type described in
Section~\ref{ssec:meshless-quadrature-reconstruction}
matches (or exceeds) the order predicted by
Proposition~\ref{prop:nystrom-classical}.
Subsequently, we want to show that the
order of convergence can be preserved even if
the node sets $X_{\h}$ and $Y_{\h}$ are decoupled,
as implied by Theorem~\ref{theo:nystrom-decoupled}.

To this aim, we consider the integral
equation \eqref{eq:fie} posed over
the Cassini oval $\Omega_1$ shown
in Figure~\ref{fig:test-domains-a},
the interior of
a quartic plane curve defined as the locus
of points such that the product of the distances
to two fixed points $(-a,0)$ and $(a,0)$,
called \emph{foci}, is a constant $b^2 \in \R_+$:
\[
\Omega_1 = \Bigl\{ (x_1,x_2) \in \R^2 \mathrel{\Big|}
\bigl((x_1+a)^2+x_2^2\bigr)
\bigl((x_1-a)^2+x_2^2\bigr) - b^4 < 0 \Bigr\}.
\]
For $a < b < a\sqrt{2}$, the domain $\Omega_1$
has smooth boundary and is simply connected,
although it is not convex: its shape resembles
the number eight rotated sideways.
For our numerical experiments, we have taken
$a = 0.95$ and $b = 1$.
Regarding the parameters of the integral equation,
we have chosen $\lambda=2$, a Gaussian kernel
with $\sigma=0.1$, Franke's function
as the exact solution $u$, and a manufactured
right-hand side $f$. As explained in
Section~\ref{sec:fie}, the norm of operator $\Kcal$
is smaller than 1, and so the choice $\lambda=2$
guarantees that the integral equation~\eqref{eq:fie}
is well-posed.

Relative errors are shown in logarithmic scale
in Figure~\ref{fig:test-choice-qW-L2-1}
as a function of discretization parameter $h$,
with $\h = (h_X,h_Y) = (h,h)$.
The identical sets of nodes $X_{\h}$ and $Y_{\h}$ are generated
by the advancing front algorithm provided by the
library NodeGenLib, see Section~\ref{ssec:decoupled-meshless},
and $h$ ranges from 0.08 to 0.02.
Different lines correspond to different
choices of quadrature orders $q_w$;
since $X_{\h}=Y_{\h}$, the Nyström method is classical
and does not involve a reconstruction operator $R_{\h}$.
We observe high orders of convergence,
with the slopes of the error curves
increasing with $q_w$, and exceeding
the expected order $q_w$ by 1 or more. This behavior
is not surprising, since it has already
been observed in \cite{davydov2025meshless}
that meshless moment-free quadrature formulas
are often superconvergent. The estimated order of convergence
(EOC) is obtained with a linear least squares fit in logarithmic scale.

In Figures~\ref{fig:test-choice-qW-L2-2} and \ref{fig:test-choice-qW-L2-3}
we switch to our decoupled Nyström scheme with $X_{\h} \neq Y_{\h}$
and introduce a reconstruction scheme with order $q_R$ equal to $q_w$.
In Figure~\ref{fig:test-choice-qW-L2-2}, $\h = (h_X,h_Y) = (h/\sqrt{2},h)$,
so that $X_{\h}$ is finer than $Y_{\h}$, whereas in the other figure
$\h = (h_X,h_Y) = (h,h/\sqrt{2})$, so that $Y_{\h}$ is finer instead.
We observe that refinement of $X_{\h}$ leads to no improvement
over the classical scheme with $X_{\h}=Y_{\h}$, but also no loss
in accuracy due to decoupling, whereas refinement of $Y_{\h}$
reduces the relative $L^2$ error by $\sqrt{2}^{q_w}$ or more, in line
with our convergence analysis \eqref{eq:errestdec}. This asymmetric
behavior strongly suggests that the quadrature term in \eqref{eq:errestdec}
dominates the reconstruction term, and so the optimal decoupling ratio $\gamma^*$
is significantly larger than 1, to the point where any reduction in reconstruction
error by lowering $h_X$ has no noticeable effect on $e_2$. This interpretation
will be confirmed by additional numerical experiments in this section.

\begin{figure}[tbp]
\centering
\hspace*{\fill}
\subcaptionbox{Classical Nyström, $X_{\h} = Y_{\h}$
with $h_X = h_Y = h$.
\label{fig:test-choice-qW-L2-1}}{
\includegraphics[width=0.3\textwidth]{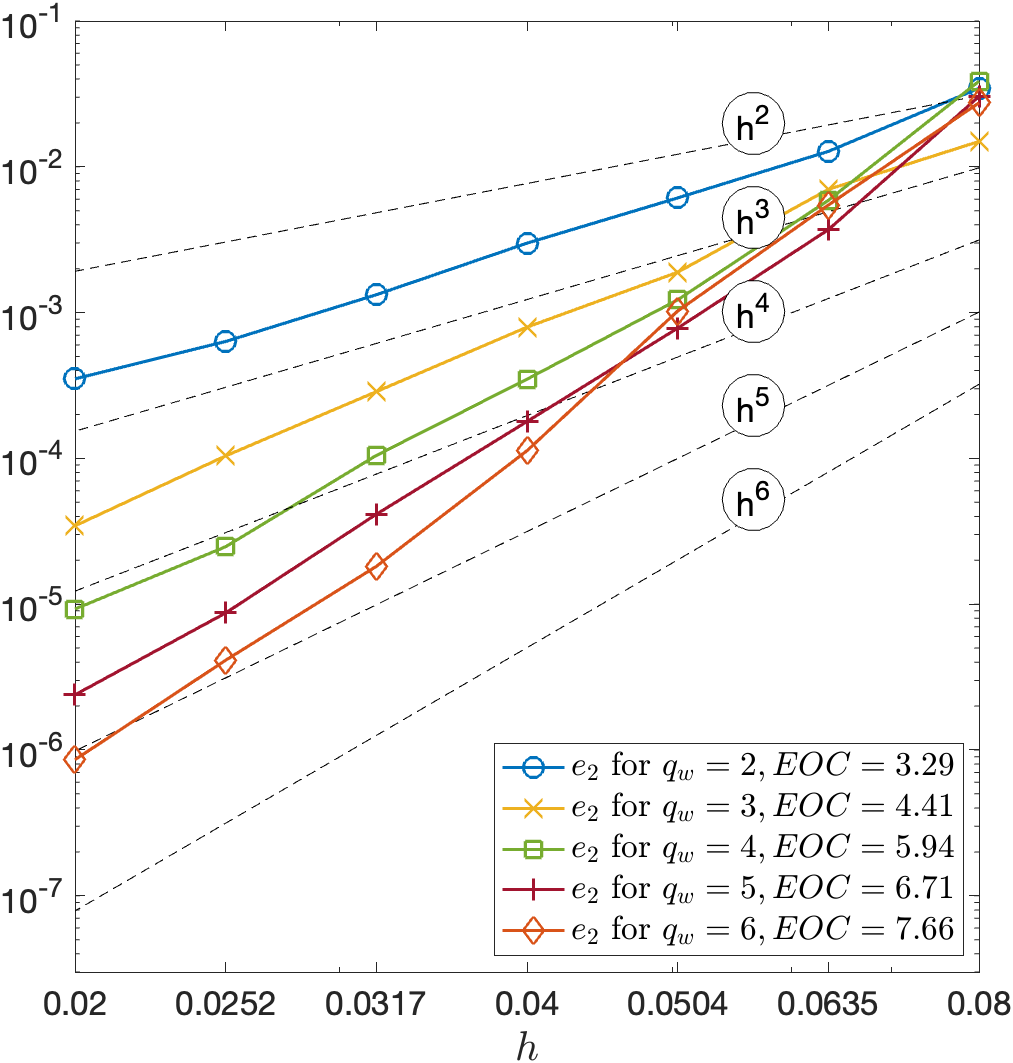}
} \hspace*{\fill}
\subcaptionbox{Decoupled Nyström, $h_X = h / \sqrt{2}$,
$h_Y = h$, $q_R = q_w$.
\label{fig:test-choice-qW-L2-2}}{
\includegraphics[width=0.3\textwidth]{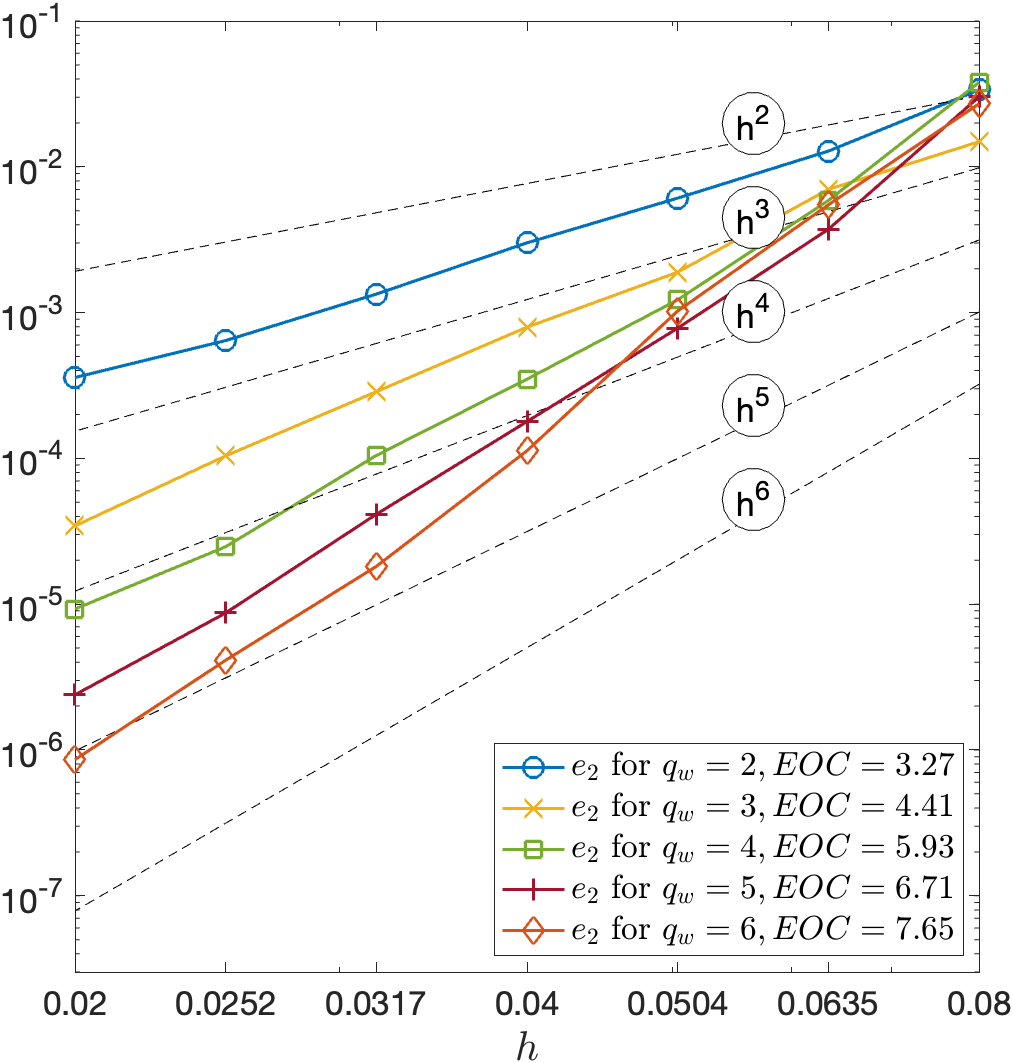}
} \hspace*{\fill}
\subcaptionbox{Decoupled Nyström, $h_X = h$,
$h_Y = h/\sqrt{2}$, $q_R = q_w$.
\label{fig:test-choice-qW-L2-3}}{
\includegraphics[width=0.3\textwidth]{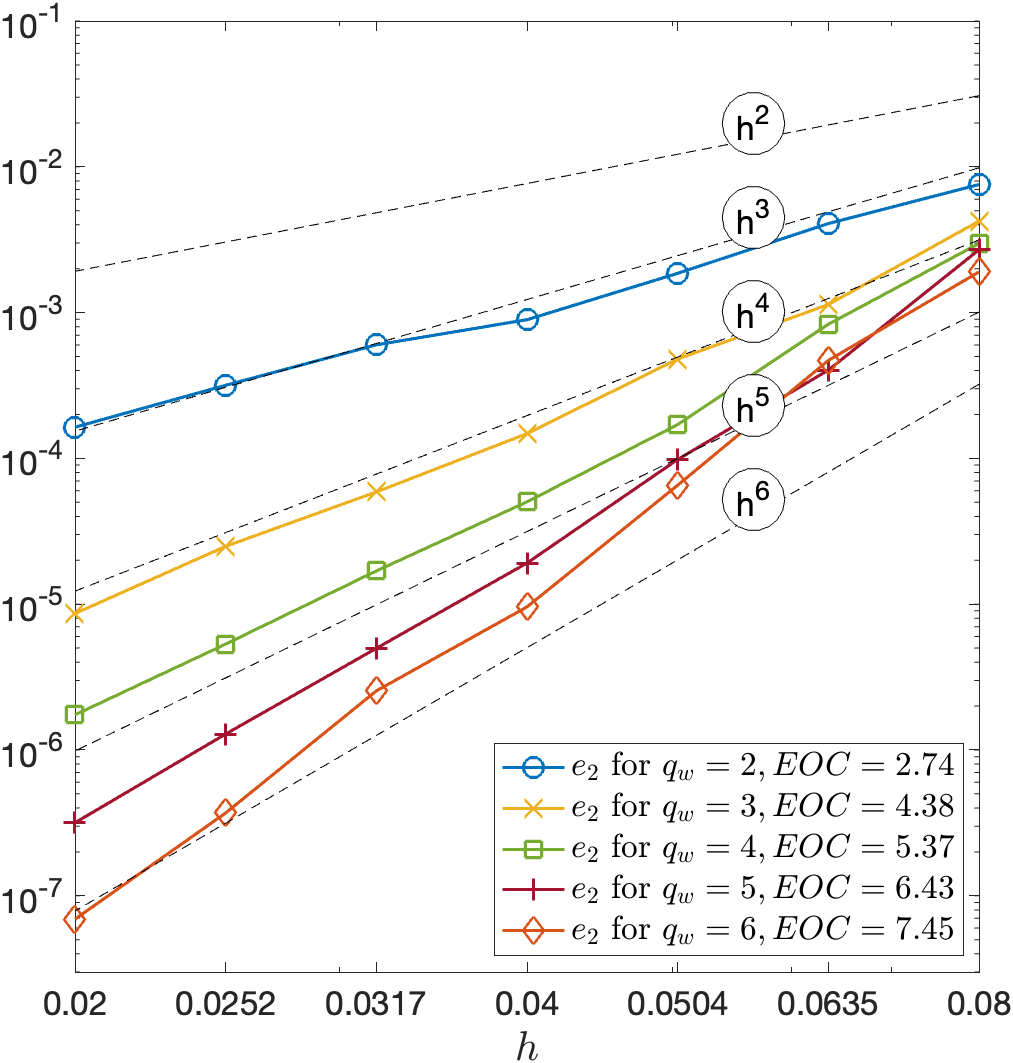}
} \hspace*{\fill}
\caption{Relative $L^2$ errors of the classical and decoupled Nyström method
used to solve the Fredholm integral equation \eqref{eq:fie} with $\lambda=2$
and Gaussian kernel with $\sigma=0.1$ over a Cassini oval. Exact solution
is Franke's function.}
\label{fig:test-choice-qW}
\end{figure}

To validate the robustness of our decoupled meshless Nyström scheme,
we switch for the next numerical experiment to the more complex
domain $\Omega_2$ shown in Figure~\ref{fig:test-domains-b},
featuring two holes and nonconvex piecewise smooth boundary.
Such a domain cannot be easily partitioned into curvilinear quadrilateral patches
required by tensor-product Gaussian quadrature, and this highlights the
geometric flexibility of our approach compared to other meshless schemes
for Fredholm integral equations \cite{mirzaei2010meshless}. We generate node sets $X_{\h}$
with $h_X = h$ and $Y_{\h}$ with $h_Y = h/\sqrt{2}$ over $\Omega_2$
using the advancing front method, and then solve the same Fredholm integral
equation as in the tests of Figure~\ref{fig:test-choice-qW},
choosing $q_w = 4$ and the Franke function as manufactured solution.
Table~\ref{tab:stability-advfront} reports the sizes $\abs{X_{\h}},\abs{Y_{\h}}$
of the sets $X_{\h},Y_{\h}$, the condition number of the system matrix $A_{\h}$,
the 1-norm of the quadrature weights $w_{\h}$ normalized by the area of $\Omega_2$,
the $\infty$-norm of the reconstruction matrix $R_{\h}$, and relative
errors in various norms. All quantities related to numerical stability
remain bounded as $h$ gets smaller. The relative error in the $\infty$-norm
decays at essentially the same rate as $e_{\text{RMS}}$ and $e_2$, with
an estimated order of convergence of 5.5.

\begin{table}[tbp]
\caption{Decoupled Nyström on nodes given by advancing front
generation. Fredholm integral equation \eqref{eq:fie} with $\lambda=2$ and
Gaussian kernel with $\sigma=0.1$ over domain $\Omega_2$.
Franke function as manufactured solution.}
\label{tab:stability-advfront}
\centering
\begin{small}
\begin{filecontents*}{test-stability-advfront.csv}
nodes_type,hX,hY,seedX,seedY,NX,NY,qW,qR,equation_type,kernel_type,sigma,amplitudep,sigmap,amplitudem,sigmam,condA,stabw,stabR,tnodegen,tquadrature,tinterpolation,tmanufactured,tassembly,tsolver,ttotal,errL1,errRMS,errL2,errLinf
AdvFront,0.08,0.0565685424949238,1,1,272,521,4,4,Fredholm,Gaussian,0.1,NaN,NaN,NaN,NaN,2.13895381338923,1.2346682524624,7.25326906459587,0.484818458333333,6.24361183333333,0.275820208333333,3.850738,0.01927425,0.003827125,10.878089875,0.00260133395978682,0.0033614959287034,0.00372301449648274,0.0093896192831662
AdvFront,0.063496042078728,0.0448984819323749,1,1,422,803,4,4,Fredholm,Gaussian,0.1,NaN,NaN,NaN,NaN,2.40913307333139,1.89622581096724,9.88551913491078,0.138974958333333,0.369702375,0.076392,3.22279620833333,0.0137410416666667,0.003726125,3.82533270833333,0.00140829101248012,0.00175987274570112,0.00187640336859881,0.00637632899182094
AdvFront,0.0503968419957949,0.0356359487256136,1,1,646,1260,4,4,Fredholm,Gaussian,0.1,NaN,NaN,NaN,NaN,1.93889758194854,1.23202532937082,9.54557048493398,0.154375958333333,0.427012666666667,0.0713945416666667,4.13701170833333,0.0267085,0.0204960833333333,4.83699945833333,0.000386694272712697,0.00041945679893172,0.000428651231482487,0.000818553017199795
AdvFront,0.04,0.0282842712474619,1,1,1011,1977,4,4,Fredholm,Gaussian,0.1,NaN,NaN,NaN,NaN,2.0432197524479,1.24487753137255,6.22348761717964,0.219500833333333,0.620368708333333,0.372413333333333,6.05709375,0.04437375,0.0221582083333333,7.33590858333333,0.000144106762968078,0.00015515938494862,0.000156650408247363,0.000206766198051465
AdvFront,0.031748021039364,0.0224492409661875,1,1,1564,3084,4,4,Fredholm,Gaussian,0.1,NaN,NaN,NaN,NaN,2.16289668087859,1.66457127353618,6.4218801466439,0.315698291666667,0.415476416666667,0.34777275,8.78716429166667,0.12106425,0.0424250833333333,10.0296010833333,8.56283462440871e-05,8.56630236133127e-05,8.61203394762432e-05,9.04743869274755e-05
AdvFront,0.0251984209978975,0.0178179743628068,1,1,2460,4852,4,4,Fredholm,Gaussian,0.1,NaN,NaN,NaN,NaN,1.91037724253842,1.22503246203592,8.04541241670548,0.500190375,1.26046625,0.198131208333333,13.7427577916667,0.285209833333333,0.111007958333333,16.0977634166667,2.66397743313533e-05,2.75733173894616e-05,2.75835531840058e-05,3.14341562446875e-05
AdvFront,0.02,0.014142135623731,1,1,3865,7630,4,4,Fredholm,Gaussian,0.1,NaN,NaN,NaN,NaN,1.91119113757787,1.29555678557883,7.33464228907568,0.762202666666667,1.69384204166667,0.255590083333333,21.562438875,0.696793958333333,0.377449125,25.34831675,1.72123644808358e-06,2.06602125464617e-06,2.0476411407552e-06,4.43199554736378e-06
\end{filecontents*}
\pgfplotstabletypeset[
	columns={
    hX,hY,NX,NY,condA,stabw,stabR,errLinf,errRMS,errL2
	}]
{test-stability-advfront.csv}
\end{small}
\end{table}

The advancing front method produces well-spaced node sets, but our
decoupled meshless Nyström scheme is also stable for node distributions
with lower regularity. Table~\ref{tab:stability-halton} reports
the same information as Table~\ref{tab:stability-advfront}, with the only
difference that rejection sampling based on Halton points is used
instead of the advancing front method. A typical node distribution produced
by rejection sampling of Halton points is shown in Figure~\ref{fig:test-domains-c},
see Section~\ref{ssec:decoupled-meshless} for more details.
As expected, nodes with lower regularity lead to larger condition numbers
and stability constants, even though all quantities related to stability
of the method nevertheless remain bounded in Table~\ref{tab:stability-halton}
for $h \to 0$. Accuracy in the case of Halton points is comparable to
advancing front, if not even greater, with the exception of $h_X = 0.02$.

\begin{table}[tbp]
\caption{Decoupled Nyström on nodes given by rejection sampling
of Halton points. Fredholm integral equation \eqref{eq:fie} with $\lambda=2$ and
Gaussian kernel with $\sigma=0.1$ over domain $\Omega_2$.
Franke function as manufactured solution.}
\label{tab:stability-halton}
\centering
\begin{small}
\begin{filecontents*}{test-stability-halton.csv}
nodes_type,hX,hY,seedX,seedY,NX,NY,qW,qR,equation_type,kernel_type,sigma,amplitudep,sigmap,amplitudem,sigmam,condA,stabw,stabR,tnodegen,tquadrature,tinterpolation,tmanufactured,tassembly,tsolver,ttotal,errL1,errRMS,errL2,errLinf
Halton,0.08,0.0565685424949238,6,7,309,573,4,4,Fredholm,Gaussian,0.1,NaN,NaN,NaN,NaN,6.69318062307144,1.88899778305128,22.4953975134379,0.0216433333333333,0.249460833333333,0.0240229583333333,1.311409,0.00565458333333333,0.0023655,1.61455620833333,0.00309396478993735,0.00386071359675614,0.00425181411746625,0.00753087701634673
Halton,0.063496042078728,0.0448984819323749,6,7,464,868,4,4,Fredholm,Gaussian,0.1,NaN,NaN,NaN,NaN,3.02484439826115,2.21801295439552,16.099200131916,0.0198549166666667,0.317132375,0.033257625,1.94692695833333,0.010898625,0.0106351666666667,2.33870566666667,0.000787461145664848,0.000976381808343573,0.000989212596964944,0.00227620201367755
Halton,0.0503968419957949,0.0356359487256136,6,7,695,1317,4,4,Fredholm,Gaussian,0.1,NaN,NaN,NaN,NaN,3.36903901285264,2.07094548043367,20.9863918113393,0.018599375,0.495678166666667,0.0492985833333333,3.073448875,0.0218186666666667,0.0139924583333333,3.672836125,0.000337482580208437,0.0003900683428099,0.000394739962441798,0.000731684661404755
Halton,0.04,0.0282842712474619,6,7,1062,2028,4,4,Fredholm,Gaussian,0.1,NaN,NaN,NaN,NaN,2.80110496771842,1.94202367614879,30.7652713258435,0.0289758333333333,0.912044625,0.0699305833333333,4.459529875,0.04718925,0.0266062916666667,5.54427645833333,7.59013827830296e-05,8.62003782460954e-05,8.35671111380624e-05,0.000134851299236124
Halton,0.031748021039364,0.0224492409661875,6,7,1626,3124,4,4,Fredholm,Gaussian,0.1,NaN,NaN,NaN,NaN,2.71564052488797,1.98007594935302,16.5964224894298,0.0303318333333333,1.53372120833333,0.099181,6.84875733333333,0.129278541666667,0.0573145416666667,8.69858445833333,2.87594174133931e-05,3.50801972329607e-05,3.54611756276253e-05,6.167060592656e-05
Halton,0.0251984209978975,0.0178179743628068,6,7,2498,4854,4,4,Fredholm,Gaussian,0.1,NaN,NaN,NaN,NaN,3.2951819584383,2.17122796453546,14.7403802680779,0.0489187916666667,2.59288395833333,0.160322708333333,10.527881625,0.292452958333333,0.127057416666667,13.7495174583333,1.56497841050331e-05,1.83909938534407e-05,1.89440360592148e-05,3.30680292212854e-05
Halton,0.02,0.014142135623731,6,7,3876,7555,4,4,Fredholm,Gaussian,0.1,NaN,NaN,NaN,NaN,4.42250617400178,1.90078372732711,21.0549708314918,0.0648610416666667,4.72055845833333,0.264511958333333,16.3328223333333,0.714324333333333,0.410400458333333,22.5074785833333,3.05892362076635e-06,3.80839281526835e-06,3.7503822279019e-06,7.28597207815084e-06
\end{filecontents*}
\pgfplotstabletypeset[
	columns={
    hX,hY,NX,NY,condA,stabw,stabR,errLinf,errRMS,errL2
	}]
{test-stability-halton.csv}
\end{small}
\end{table}

Going back to the hypothesis that the numerical errors in
Figure~\ref{fig:test-choice-qW} are saturated by quadrature errors
and not by the size of the solution set $X_{\h}$ or the accuracy
of the reconstruction scheme, we have tried lowering $q_R$ while
keeping $q_w = 6$ and choosing all the other parameters
(including the domain) as in the tests of Figure~\ref{fig:test-choice-qW}.
The outcome, shown in Figure~\ref{fig:test-choice-qR-1}, is that
reconstruction order can be reduced heavily before the numerical errors
start to be affected, further supporting our interpretation based on
estimate~\eqref{eq:errestdec}. Looking at the values of $h$ where the lines
corresponding to $q_R = 2,3,4$ split from $q_R = 5,6$, we can infer
that quadrature and reconstruction terms in the error are roughly balanced
around $h = 0.08$ for $q_R = 2$, around $h = 0.04$ for $q_R = 3$,
and around $h = 0.02$ for $q_R = 4$. For smaller values of $h$,
the reconstruction error (with lower order of convergence)
dominates the quadrature error (with higher order of convergence).
More insight on the convergence regime where reconstruction errors
dominate quadrature errors is given by Figure~\ref{fig:test-choice-qR-2},
where we consider all possible $q_w$ from 2 to 6 while keeping $q_R = 2$.
As expected, the numerical error is saturated by reconstruction accuracy
and converges to zero quadratically even for $q_w > 2$.

\begin{figure}[tbp]
\centering
\hspace*{\fill}
\subcaptionbox{Fixed quadrature order $q_w = 6$,
variable reconstruction order $q_R$.
\label{fig:test-choice-qR-1}}{
\includegraphics[width=0.4\textwidth]{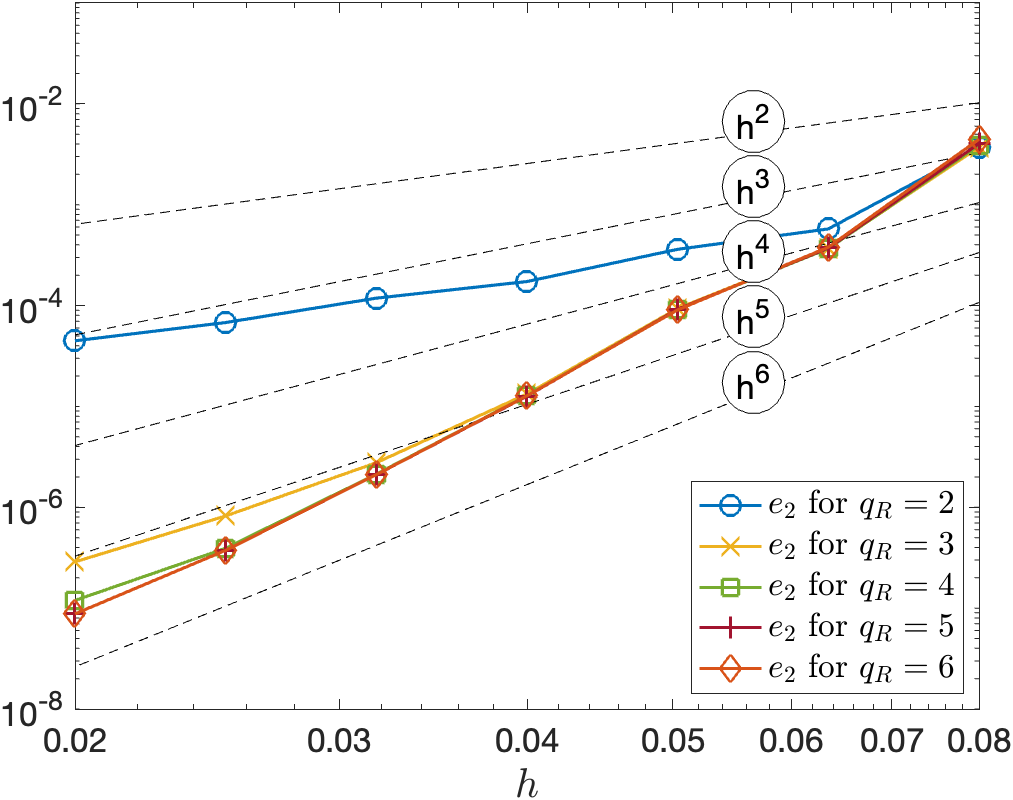}
} \hspace*{\fill}
\subcaptionbox{Fixed reconstruction order $q_R = 2$,
variable quadrature order $q_w$.
\label{fig:test-choice-qR-2}}{
\includegraphics[width=0.4\textwidth]{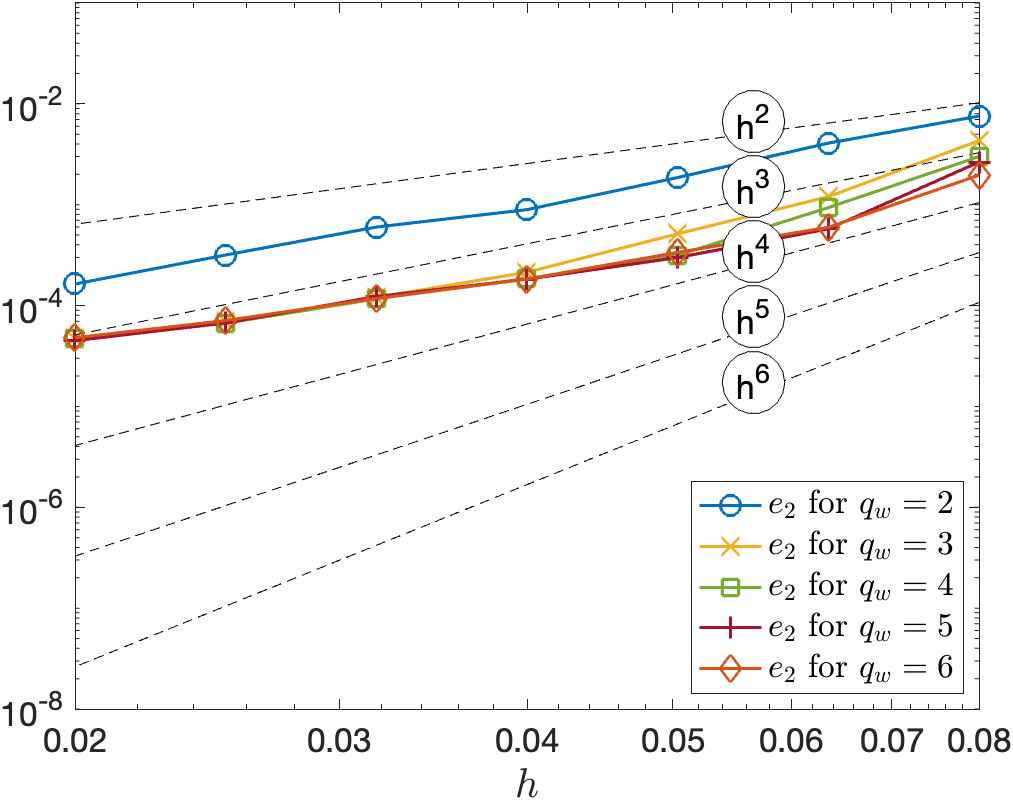}
} \hspace*{\fill}
\caption{Decoupled Nyström method with $h_X = h$ and
    $h_Y = h/\sqrt{2}$. Fredholm integral equation
    \eqref{eq:fie} with $\lambda=2$ and Gaussian kernel
    with $\sigma=0.1$ over a Cassini oval. Exact solution is
    Franke's function.}
\label{fig:test-choice-qR}
\end{figure}

At this point, two issues naturally arise. First, we need to determine how
much $h_Y$ can be reduced for a fixed $h_X$ before the numerical error
is dominated by the reconstruction term, meaning that $X_{\h}$ is too
coarse compared to $Y_{\h}$.
Second, we want to evaluate the computational efficiency gained in
a Nyström scheme by decoupling $X_{\h}$ from $Y_{\h}$.
We address these points in the next section.

\subsection{On the choice of decoupling ratio}
The theoretical analysis in the proof of Theorem~\ref{theo:nystrom-decoupled}
and the numerical experiments performed in Section~\ref{ssec:validation-convergence-order}
support one of the key conclusions in our work, namely
that the numerical error of a decoupled Nyström scheme can be
understood as the sum of a quadrature term and a reconstruction term.
We have seen that, whenever one of the two terms dominates
the other, a further reduction of the dominated term brings no benefit,
but only the overhead of additional computation. Equivalently,
we could say the increase of the dominated term reduces the overall
computational cost of the Nyström scheme without any significant loss in accuracy.
For this reason, a decoupled Nyström scheme is just as efficient
as a classical one whenever the optimal decoupling ratio
$\gamma^*$ as defined in \eqref{eq:gamma-star} is close to 1,
and can be much more efficient for problems where $\gamma^* \gg 1$.
It follows from \eqref{eq:gamma-star} that $\gamma^* \gg 1$ whenever
\[
\norm{\norm{k(x,y)}_{C^{q}(\Omegabar), y}}_{\infty, x}
\gg \norm{k}_{C^0(\Omegabar \times \Omegabar)},
\]
which is typical of Fredholm integral equations with a narrow
kernel $k$, such as Gaussian kernels with small $\sigma$.

Since the optimal decoupling ratio $\gamma^*$ is hard to estimate
a priori but is independent of $h_X$ and $h_Y$,
we can approximate it for a particular problem
using coarse sets $X_{\h}$ and $Y_{\h}$, where numerical solutions
can be computed quickly, and then use this guess for finer node sets,
where the solution can be found with the full required accuracy.
To verify this claim, we have plotted relative $L^2$ errors
as functions of decoupling ratios $\gamma$ in
Figure~\ref{fig:test-find-gamma}. In each subfigure, the two curves
are associated with different sets of solution nodes $X_{\h}$:
the top curve with coarse nodes generated by the advancing front
method with $h_X = 0.08$, and the bottom curve with finer nodes
with $h_X = 0.04$.

As $\gamma$ gets larger, the numerical errors converge to a finite
value, which is essentially the reconstruction error (since the
quadrature error becomes negligible). The values of $\gamma$ that
balance the quadrature and reconstruction errors correspond
to the points in the graphs where the curves level off: approximately
$\gamma = 2$ for $\sigma = 0.2$, $\gamma = 3$ for $\sigma = 0.1$,
and $\gamma = 5$ for $\sigma = 0.05$. Note how the gap between the
two curves remains constant in logarithmic scale, which confirms that
a nearly optimal decoupling ratio $\gamma^*$ can be found on a coarse
set of nodes, once the Fredholm integral equation has been fixed,
in particular its kernel~$k$. This empirically verifies that
$\gamma^*$ depends on $\sigma$ but not on $h_X$, as predicted
by~\eqref{eq:gamma-star}.

\begin{figure}[tbp]
\centering
\hspace*{\fill}
\subcaptionbox{Gaussian kernel with $\sigma=0.2$.
\label{fig:test-find-gamma-3}}{
\includegraphics[width=0.3\textwidth]{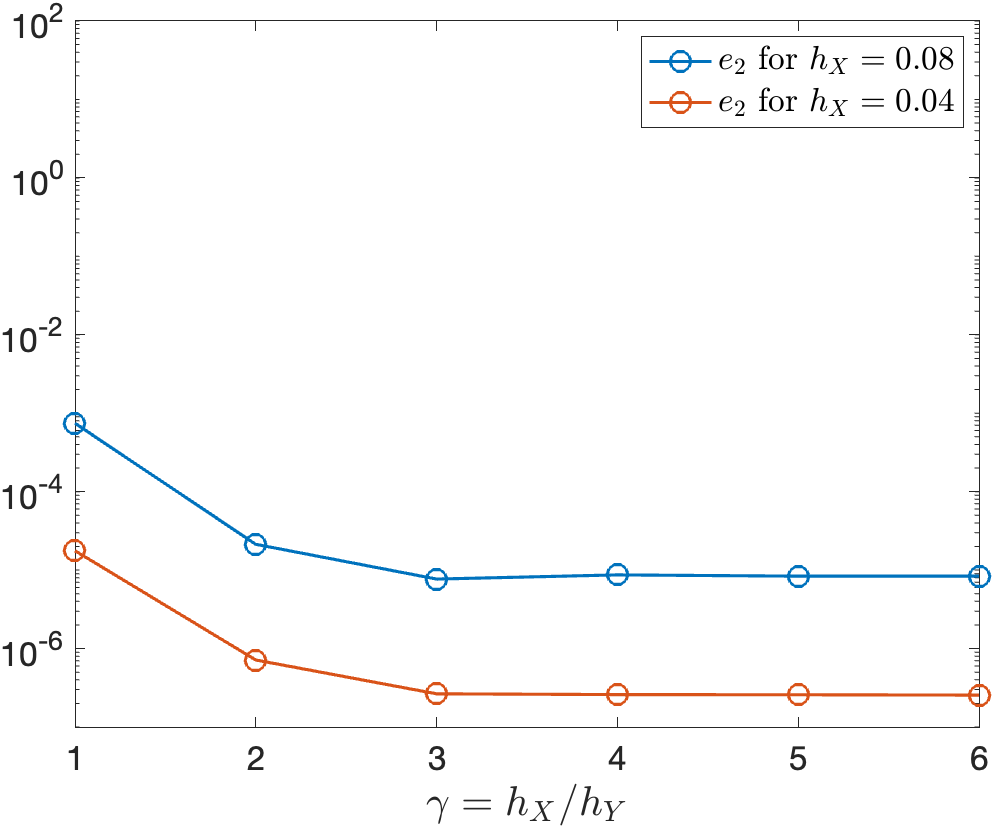}
} \hspace*{\fill}
\subcaptionbox{Gaussian kernel with $\sigma=0.1$.
\label{fig:test-find-gamma-2}}{
\includegraphics[width=0.3\textwidth]{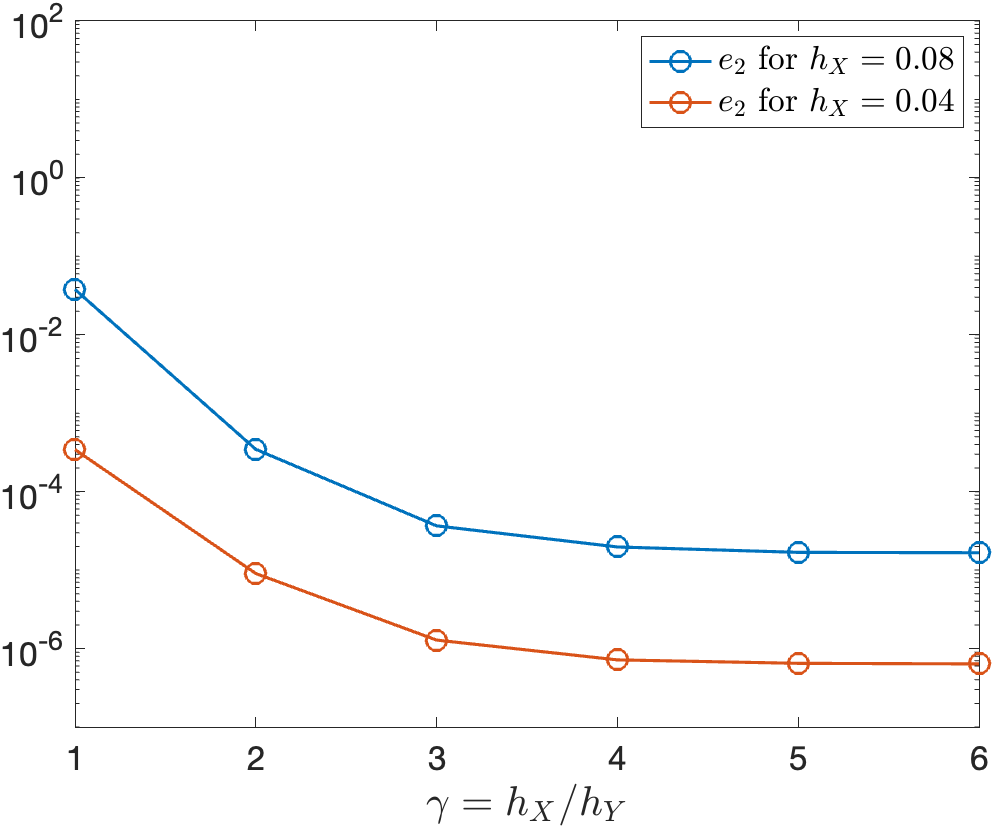}
} \hspace*{\fill}
\subcaptionbox{Gaussian kernel with $\sigma=0.05$.
\label{fig:test-find-gamma-1}}{
\includegraphics[width=0.3\textwidth]{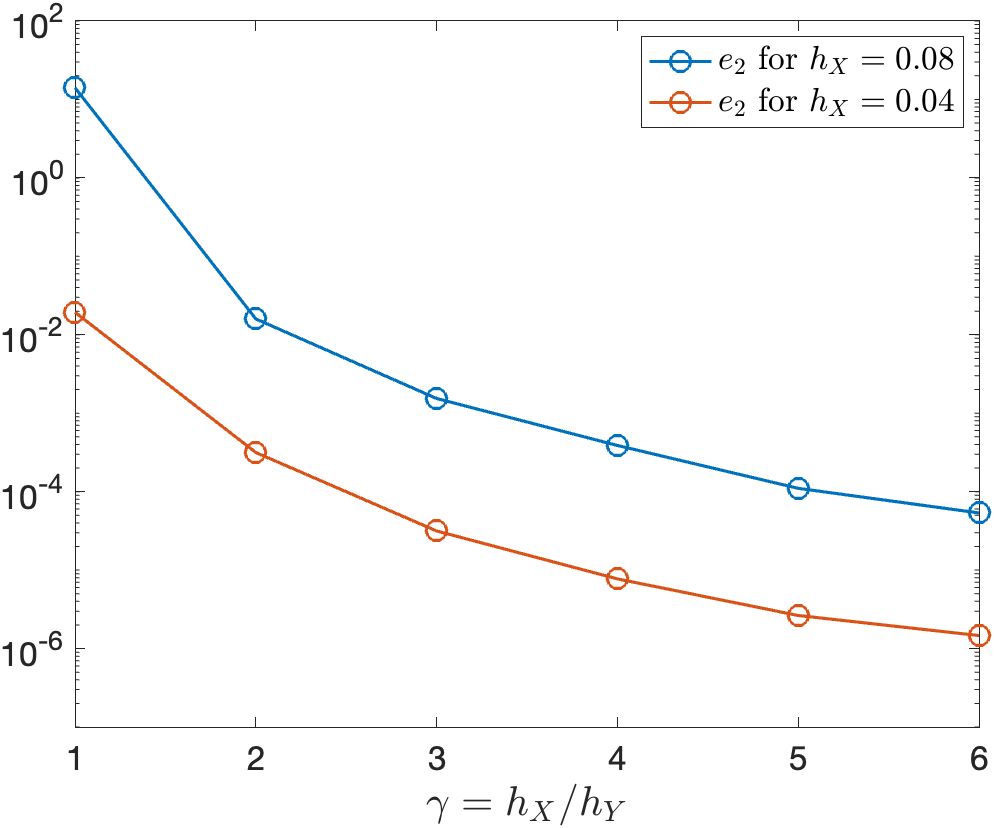}
} \hspace*{\fill}
\caption{Fredholm integral equation~\eqref{eq:fie} with $\lambda=2$
over the Cassini oval $\Omega_1$. Franke's function as manufactured solution.
Relative $L^2$ errors as a function of the decoupling ratio $\gamma$
for $q = q_w = q_R = 4$.}
\label{fig:test-find-gamma}
\end{figure}

A decoupling ratio $\gamma$ may be optimal in the sense
of~\eqref{eq:gamma-star},
but in practice the best choice of $\gamma$ is the one
that minimizes the product of the error
of the decoupled Nyström scheme in some norm and the time required
to compute the numerical solution, including not only the cost
of assembling and solving system~\eqref{eq:nystrom-decoupled},
but also the cost of computing the node sets $X_{\h}, Y_{\h}$,
the reconstruction matrices $R_{\h}$ and
the quadrature weights $\w_{\h}$.
In the following, we refer to the inverse of this product
as the \emph{efficiency} of the Nyström scheme.

Figure~\ref{fig:test-choice-hYhX-efficiency} displays relative
$L^2$ errors as functions of computation time, excluding from
the runtime only the cost of evaluating the manufactured
right-hand side $f$ (which relies on Chebfun and adaptive Gaussian
quadrature on $\partial\Omega$, as explained at the beginning of
Section~\ref{sec:numerical-experiments}).
As before, the Fredholm integral equation~\eqref{eq:fie}
with $\lambda=2$ is solved over the Cassini oval $\Omega_1$
choosing Franke's function as the manufactured solution.
Seven values of $h_X$ uniformly spaced over $[0.01,0.08]$
in logarithmic scale are used, even though not all possible
combinations of $h_X$ and $\gamma$ are included in the plots
to reduce clutter and out-of-range values.
For the considered problems with $\sigma = 0.1$ or $\sigma=0.05$,
our decoupled meshless Nyström scheme is one to two orders of magnitude
more efficient than the classical Nyström scheme with $\gamma=1$ and $X_{\h} = Y_{\h}$.
The optimal choice of $\gamma$ from the point of view of efficiency
is reasonably close to the estimate of $\gamma^*$ that can be obtained
using coarse node sets. Note that refinement of quadrature nodes eventually
leads to diminishing returns, even for $\gamma$ smaller than $\gamma^*$:
Figure~\ref{fig:test-find-gamma-1} with $\sigma=0.05$ suggests that
there is something to be gained by considering decoupling ratios
$\gamma > 3$, but the efficiency analysis in
Figure~\ref{fig:test-choice-hYhX-efficiency-2} actually shows that the
increase in accuracy is almost entirely offset by the increase in computation time.
Clearly, this depends on the level of optimization of quadrature-related
code compared to the rest.

Aside from external libraries such as NodeGenLib, our code is written in MATLAB R2024b
and the execution time was measured on a MacBook Air with Apple M1 processor
and 8GB of RAM. To reduce the impact of thermal throttling,
a warmup run was executed before collecting performance data.
The dense linear systems were solved using MATLAB's backslash operator.
The last tests in the $\gamma=1$ lines are the only ones
with $h_X = 0.01$, which lead to a solution set $X_{\h}$ of size
20689 and a linear system of size $20689 \times 20689$ to be solved:
storing the system matrix alone in double precision requires more than 3GB of memory.
The sharp decrease in efficiency for these tests is therefore
due to excessive memory usage, which causes swapping, i.e.~the transfer
of working memory from RAM to permanent storage and back.

\begin{figure}[tbp]
\centering
\hspace*{\fill}
\subcaptionbox{Gaussian kernel with $\sigma=0.1$.
\label{fig:test-choice-hYhX-efficiency-1}}{
\includegraphics[width=0.4\textwidth]{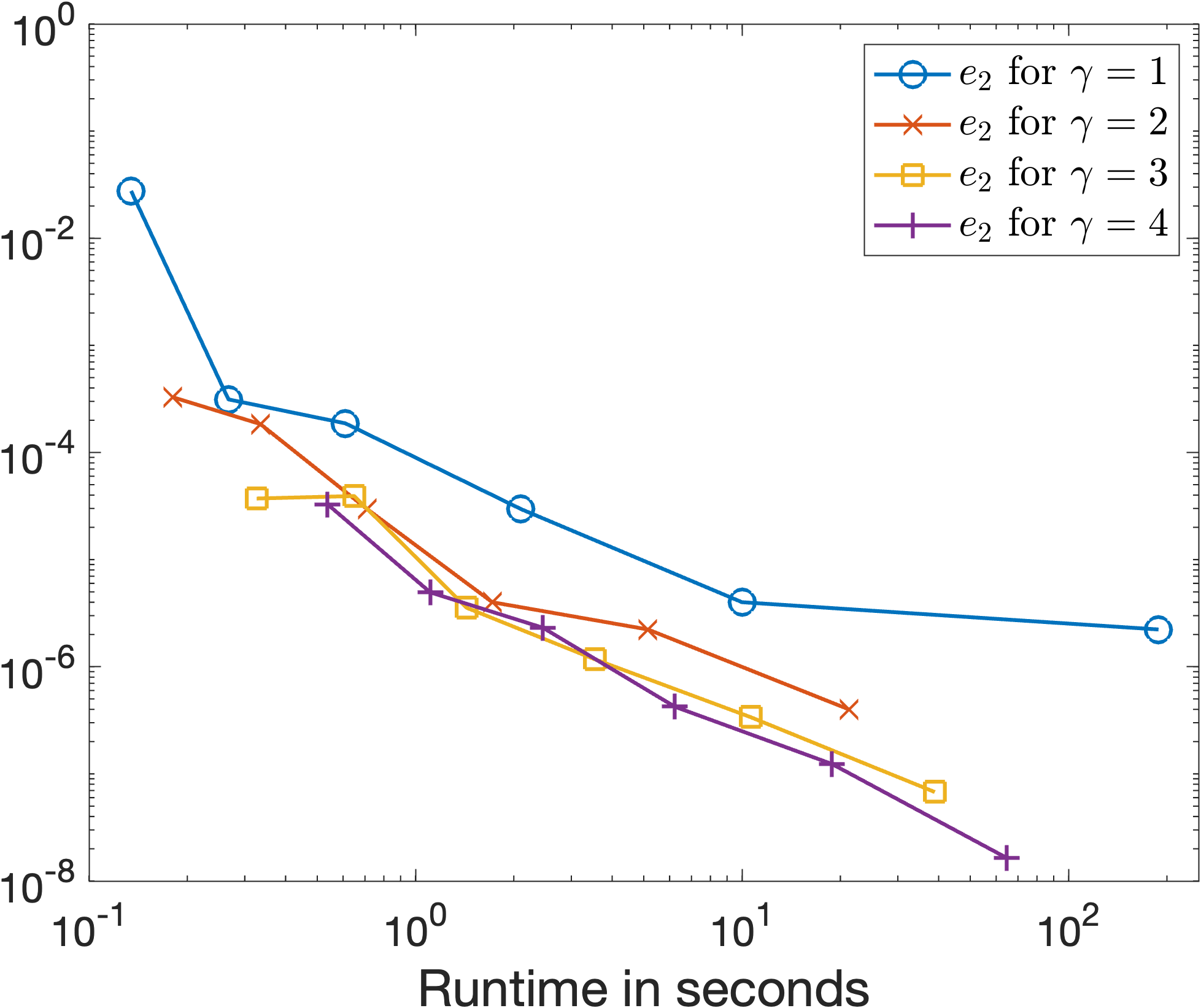}
} \hspace*{\fill}
\subcaptionbox{Gaussian kernel with $\sigma=0.05$.
\label{fig:test-choice-hYhX-efficiency-2}}{
\includegraphics[width=0.4\textwidth]{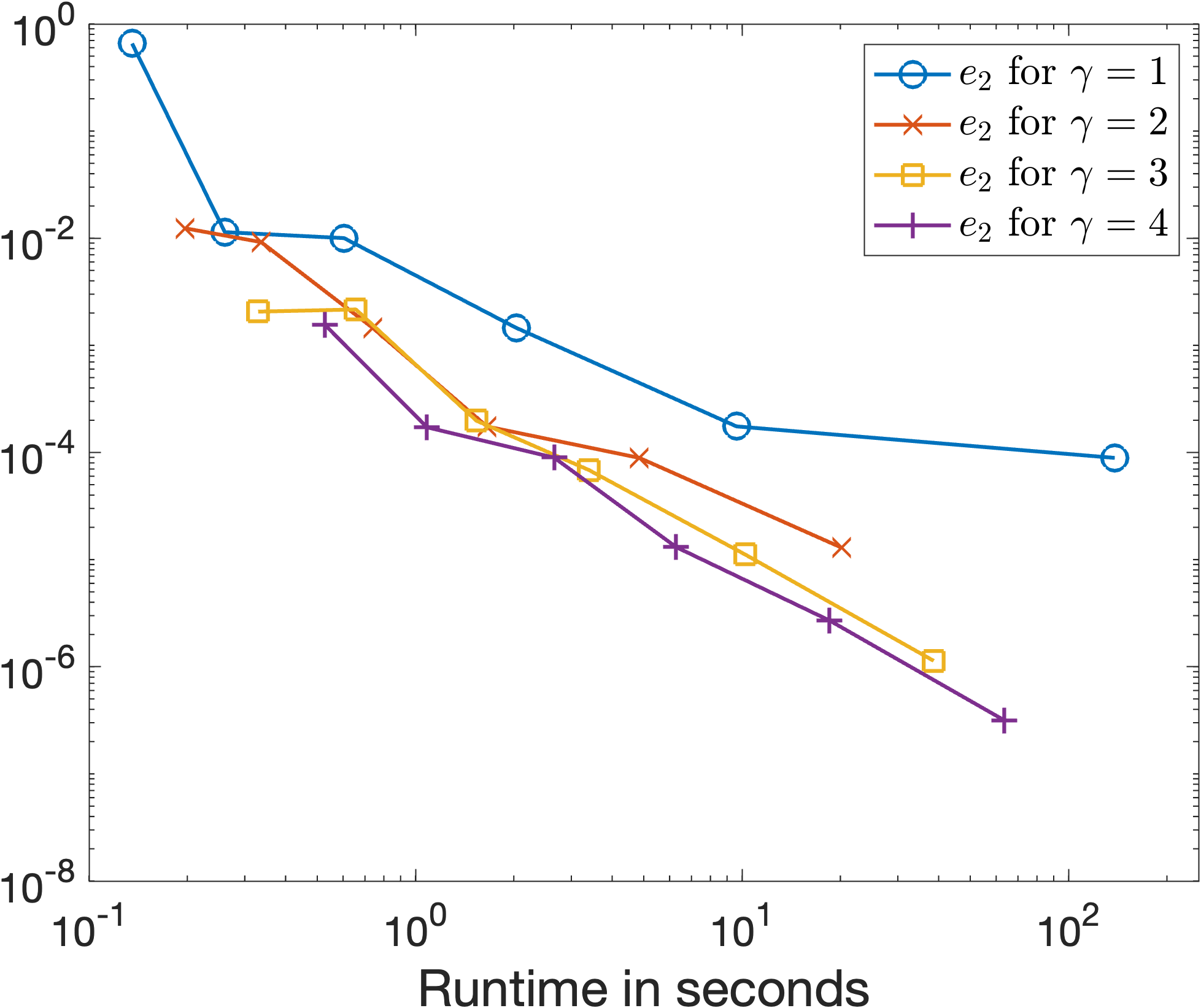}
} \hspace*{\fill}
\caption{Fredholm integral equation~\eqref{eq:fie} with $\lambda=2$
over the Cassini oval $\Omega_1$. Franke's function as manufactured solution.
Relative $L^2$ errors as a function of total computation time
for $q = q_w = q_R = 4$.}
\label{fig:test-choice-hYhX-efficiency}
\end{figure}

%% file: section5.tex
\section{Conclusion and future work}
\label{sec:conclusion}

In this paper, we have generalized the classical Nyström method
for the solution of~\eqref{eq:fie} by decoupling the set of solution nodes
from the set of quadrature nodes used to discretize the integral
operator $\Kcal$ with a smooth kernel $k$.
The combined use of local RBF interpolation with polyharmonic splines
and meshless moment-free quadrature formulas reduces the complexity
of the implementation by relying on the same building block,
namely the numerical differentiation formulas~\eqref{eq:differentiation-formula},
and leads to a meshless high-order method suitable for arbitrarily
shaped domains that is stable even for irregular distributions of nodes.
The domain $\Omega_2$ in Figure~\ref{fig:test-domains} is among the
most complex on which a meshless Nyström method has been tested in the literature.

Our decoupled scheme was found to be significantly more efficient
than the classical Nyström method, especially for problems with narrow kernels.
We have shown through theoretical error estimates and targeted
numerical experiments that the error of our new scheme can be
understood as the sum of a reconstruction error that scales with
the density of solution nodes and a quadrature error that scales
with the density of quadrature nodes.
In practice, we recommend balancing the orders of reconstruction
and quadrature ($q_R = q_w$) and tuning the decoupling ratio $\gamma$
by a parameter sweep using coarse sets of nodes.
We stress that our method, just like the classical Nyström scheme,
produces an approximation of the solution everywhere in~$\Omega$
via the Nyström interpolation formula~\eqref{eq:interp-decoupled}.
To validate convergence rates, we have introduced a novel technique
to construct manufactured solutions on piecewise smooth 2D domains,
assuming that a parametric description of their boundaries is known.

While this study establishes the foundation for decoupled Nyström
schemes, further research is needed to fully understand the
opportunities and advantages offered by this generalization.
For example, now that the set of quadrature nodes $Y_{\h}$
has been decoupled from the solution nodes $X_{\h}$,
our Nyström scheme could be extended to singular kernels by locally
refining $Y_{\h}$ around each element of $X_{\h}$.
Fredholm integral equations \eqref{eq:fie} often arise in applications
as equilibrium states of more complex integro-differential models,
possibly including nonlinear terms. An extension of our numerical
method to such models is the subject of ongoing work.
Since the spectrum of $\Kcal$ is also of interest in some applications,
it would be worth investigating how well it can be approximated
by the eigenvalues of the matrix $K_{\h} W_{\h} R_{\h}$.
Note that the operator $\Kcal$ amounts to a convolution with the kernel $k$,
so its discretization can be used to filter scattered data over
an arbitrary bounded domain. In the case of Gaussian kernels,
one obtains a family of low-pass filters indexed by the smoothing
parameter $\sigma > 0$.

In this study, numerical experiments were only performed in 2D,
but all the techniques, including the use of moment-free meshless
quadrature formulas, readily generalize to three-dimensional
domains and surfaces embedded in $\R^3$.
Investigating the behavior of our decoupled Nyström scheme
in higher dimensions would therefore be another logical next step.
Finally, one could investigate the compressibility of $A_{\h}$
using low-rank methods, as well as the use of iterative solvers
for the linear system~\eqref{eq:nystrom-decoupled},
and whether the effectiveness of these techniques depends
on the decoupling ratio $\gamma$.

%% file: appendixA.tex
\section{Appendix on continuously differentiable functions}
\label{sec:appendix-functional-analysis}

Let $\Omega$ be an open and bounded set in $\R^d$.
For any integer $q \geq 0$
we define $C^q(\Omega)$ as the space of functions
over $\Omega$ that are continuously differentiable
up to order $q$. The derivatives of functions
$f \in C^q(\Omega)$ are conveniently described using
multi-index notation: for any $d$-tuple of non-negative
integers
\[
\alpha = (\alpha_1, \ldots, \alpha_d),
\]
the operator $\partial^\alpha f$ denotes the mixed
partial derivative obtained by differentiating
$f$ $\alpha_i$ times with respect to the $i$-th variable,
for each $i = 1, \ldots, d$. In total, this involves
$\abs{\alpha} \deq \alpha_1 + \cdots + \alpha_d$ derivatives,
and so the expression $\partial^\alpha f$ with
$f \in C^q(\Omega)$ is meaningfully defined for any multi-index
$\alpha$ such that $\abs{\alpha} \leq q$.
The sum and difference of multi-indices are defined
componentwise, although the difference $\alpha-\beta$
is only well defined if $\beta_i \leq \alpha_i$
for all $i = 1,\dots,d$, a condition which we denote
by $\beta \leq \alpha$ and that defines a partial order
on the set of multi-indices.
The factorial and binomial coefficient are defined
for multi-indices as
\[
\alpha!
= \prod_{i=1}^d \alpha_i !,
\qquad
\binom{\alpha}{\beta}
= \frac{\alpha!}{\beta!(\alpha-\beta)!}
\text{ with $\beta \leq \alpha$,}
\]
and they allow the multivariate Leibniz rule
to be elegantly stated as
\[
\partial^\alpha (fg)
= \sum_{\beta \leq \alpha} \binom{\alpha}{\beta}
\, \partial^\beta f \; \partial^{\alpha-\beta} g.
\]
By requiring each mixed partial derivative $\partial^\alpha f$
to be continuous up to the boundary of $\Omega$,
spaces of continuously differentiable functions
can be given the structure of a Banach space.

\begin{definition}
The space $C^q(\Omegabar)$ of continuously differentiable
functions up to order $q$ and up to the boundary
of $\Omega$ is defined as
\[
C^q(\Omegabar)
= \left\{ f \in C^q(\Omega) \mid \text{$\partial^\alpha f$
extends continuously to $\Omegabar$ for all
$\abs{\alpha} \leq q$} \right\}.
\]
\end{definition}

\begin{proposition}
For any $f \in C^q(\Omegabar)$, the expression
\begin{equation} \label{eq:Cq-norm}
\norm{f}_{C^q}
\deq \sum_{\abs{\alpha} \leq q} \frac{1}{\alpha!}
    \norm{\partial^\alpha f}_\infty
= \sum_{\abs{\alpha} \leq q} \frac{1}{\alpha!}
\max_{x \in \Omegabar} \abs{\partial^\alpha f(x)}
\end{equation}
defines a norm, and makes $C^q(\Omegabar)$ a Banach space.
\end{proposition}

\begin{proof}
The statement is typically proved without the
weights $1/\alpha!$, see for example \cite{evans2022partial}.
However, it is clear that including the weights
defines an equivalent norm, because the weights
are strictly positive.
\end{proof}

In the norm definition in (\ref{eq:Cq-norm}) we include the
weights $1/\alpha!$ to ensure to get a submultiplicative norm,
a property which is proved in the next proposition and that we
need in the proof of Theorem~\ref{theo:nystrom-decoupled}.

\begin{proposition} \label{prop:submult}
For any $f,g \in C^q(\Omegabar)$, we have
\[
\norm{fg}_{C^q} \leq \norm{f}_{C^q} \norm{g}_{C^q}.
\]
\end{proposition}

\begin{proof}
By the multivariate Leibniz rule and the triangle inequality,
\begin{align*}
\norm{fg}_{C^q}
&= \sum_{\abs{\alpha} \leq q}
    \frac{1}{\alpha!} \norm{\partial^\alpha(fg)}_\infty
= \sum_{\abs{\alpha} \leq q} \frac{1}{\alpha!} \normBig{
    \sum_{\beta \leq \alpha} \binom{\alpha}{\beta}
    \, \partial^\beta f \; \partial^{\alpha-\beta} g}_\infty \\
&\leq \sum_{\abs{\alpha} \leq q} \sum_{\beta \leq \alpha}
    \frac{1}{\beta! (\alpha-\beta)!}
    \norm{\partial^\beta f}_\infty
    \norm{\partial^{\alpha-\beta} g}_\infty.
\end{align*}
Summing over the multi-indices $\alpha$ and $\beta$
for $\abs{\alpha} \leq q$ and
$\beta \leq \alpha$ is equivalent to summing over the
multi-indices $\beta$ and $\gamma \deq \alpha-\beta$
for $\abs{\beta+\gamma} \leq q$, and so
\[
\sum_{\abs{\alpha} \leq q} \sum_{\beta \leq \alpha}
    \frac{1}{\beta! (\alpha-\beta)!}
    \norm{\partial^\beta f}_\infty
    \norm{\partial^{\alpha-\beta} g}_\infty
= \sum_{\beta,\gamma}^{\abs{\beta+\gamma} \leq q}
    \frac{1}{\beta! \gamma!}
    \norm{\partial^\beta f}_\infty
    \norm{\partial^\gamma g}_\infty.
\]
Estimating the sum from above by introducing the terms
with $\abs{\beta+\gamma} > q$ completes the proof:
\[
\sum_{\beta,\gamma}^{\abs{\beta+\gamma} \leq q}
    \frac{1}{\beta! \gamma!}
    \norm{\partial^\beta f}_\infty
    \norm{\partial^\gamma g}_\infty
\leq \sum_{\abs{\beta} \leq q} \sum_{\abs{\gamma} \leq q}
    \frac{1}{\beta! \gamma!}
    \norm{\partial^\beta f}_\infty
    \norm{\partial^\gamma g}_\infty
= \norm{f}_{C^q} \norm{g}_{C^q}. \qedhere
\]
\end{proof}

Relatively compact subsets of $C^0(\Omegabar)$
are characterized by the Arzelà-Ascoli theorem,
see Section~C.7 of~\cite{evans2022partial}.
Even though this result can be generalized to
spaces $C^q(\Omegabar)$ by extending the assumptions
of uniform boundedness and uniform equicontinuity to
all mixed partial derivatives up to order $q$,
we only recall in the next proposition the classical statement with $q = 0$,
which is sufficient for the proof of Theorem~\ref{theo:nystrom-decoupled}.

\begin{proposition} \label{prop:asc-arz}
A subset $\Fcal \subset C^0(\Omegabar)$
is relatively compact with respect to the $\infty$-norm
if and only if
\begin{enumerate}
\item Functions in $\Fcal$ are uniformly bounded: 
there exists $C_{\Fcal}$ such that
\[
\norm{f}_\infty \leq C_{\Fcal}
\quad \text{for all $f \in \Fcal$}.
\]
\item Functions in $\Fcal$ are uniformly equicontinuous:
for all $\varepsilon > 0$,
there exists $\delta_{\Fcal} = \delta_{\Fcal}(\varepsilon) > 0$ such that
\[
\abs{f(x)-f(x')} < \varepsilon
\quad \text{for all $f \in \Fcal$ and $x,x' \in \Omegabar$
such that $\norm{x-x'} < \delta_{\Fcal}$}.
\]
\end{enumerate}
\end{proposition}

The following lemmas lead to Proposition~\ref{prop:pointwise-conv-Cq},
which is in turn required by our proof of Theorem~\ref{theo:nystrom-decoupled}.

\begin{lemma} \label{lemma:density}
The space $C^q(\Omegabar)$ is a dense subset of $C^0(\Omegabar)$.
\end{lemma}
\begin{proof}
Any function $f \in C^0(\Omegabar)$ with $\Omega \subset \R^d$
can be extended to a function $\tilde{f} \in C^0(\R^d)$
such that
\[
\normbig{\tilde{f}}_{\infty,\R^d}
= \normbig{f}_{\infty,\Omegabar}
\]
by the Tietze extension theorem,
see \cite{munkres2000topology} for a reference.
Consider the following function $\mu \in C^{\infty}(\R^d)$,
known as \emph{standard mollifier},
\[
\mu(x) = \begin{cases}
    C_{\mu} \exp\left(\frac{1}{\norm{x}^2-1}\right)
    & \text{if $\norm{x} < 1$,} \\
    0
    & \text{if $\norm{x} \geq 1$,}
\end{cases}
\qquad \text{with constant $C_\mu > 0$ selected so that
$\int_{\R^d} \mu(x) \dx = 1$,}
\]
and define \emph{mollifiers} $\mu_\varepsilon$
by scaling $\mu$ by a factor $\varepsilon \in (0,1)$ and requiring again unit integral in $\R^d$.
Let $\tilde{f}_\varepsilon$ be the convolution of $\tilde{f}$
and $\mu_\varepsilon$:
\[
\mu_\varepsilon(x) \deq \varepsilon^{-d} \mu(x/\varepsilon),
\quad x \in \R^d,
\quad \tilde{f}_\varepsilon(x)
    \deq \int_{\R^d} \mu_\varepsilon(x-y) \tilde{f}(y) \dy.
\]
By the properties of mollifiers,
the functions $\tilde{f}_\varepsilon$ belong to
$C^\infty(\R^d)$ and, for $\varepsilon \to 0^+$, they 
converge uniformly to $\tilde{f}$
on every compact subset of $\R^d\,,$ 
see Section~C.4 in \cite{evans2022partial}. In particular,
they converge uniformly on $\Omegabar$, and this shows
that $C^q(\Omegabar)$ is a dense subset of $C^0(\Omegabar).$
\end{proof}

\begin{lemma} \label{lemma:strong-convergence}
A sequence $\{T_i\}_{i \in \N}$ of bounded linear operators
$T_i \colon C^0(\Omegabar) \to C^0(\Omegabar)$
converges strongly to a bounded linear operator
$T \colon C^0(\Omegabar) \to C^0(\Omegabar)$, which means that
\[
\lim_{i \to \infty} \norm{T_i f - T f}_{\infty} = 0
\quad \text{for all $f \in C^0(\Omegabar)$},
\]
if and only if
\begin{enumerate}
\item For all $f$ in a dense subset of $C^0(\Omegabar)$,
\[
\lim_{i \to \infty} \norm{T_i f - T f}_{\infty} = 0.
\]
\item There exists $C > 0$ such that $\norm{T_i} \leq C$ for all $i \in \N$.
\end{enumerate}
\end{lemma}

\begin{proof}
This is a particular case of Lemma~9.4.7 in \cite{eidelman2004functional},
where a proof is given for arbitrary Banach spaces.
\end{proof}

Combining the two lemmas, whenever we are dealing with a sequence
$\{T_i\}_{i \in \N}$ of bounded linear operators
$T_i \colon C^0(\Omegabar) \to C^0(\Omegabar)$
which are supposed to converge strongly to
$T \colon C^0(\Omegabar) \to C^0(\Omegabar)$ for $i \to \infty$,
the density of $C^q(\Omegabar)$ in $C^0(\Omegabar)$ can be used
to check convergence of $T_i$ only for functions in $C^q(\Omegabar)$,
provided that the operators $T_i$ are uniformly bounded.
We formulate this observation in the following proposition.

\begin{proposition} \label{prop:pointwise-conv-Cq}
For any integer $q \geq 0$, a sequence $\{T_i\}_{i \in \N}$
of bounded linear operators $T_i \colon C^0(\Omegabar) \to C^0(\Omegabar)$
converges strongly to a bounded linear operator
$T \colon C^0(\Omegabar) \to C^0(\Omegabar)$, which means that
\[
\lim_{i \to \infty} \norm{T_i f - T f}_{\infty} = 0
\quad \text{for all $f \in C^0(\Omegabar)$},
\]
if and only if
\begin{enumerate}
\item For all $f \in C^q(\Omegabar)$,
\[
\lim_{i \to \infty} \norm{T_i f - T f}_{\infty} = 0.
\]
\item There exists $C > 0$ such that $\norm{T_i} \leq C$ for all $i \in \N$.
\end{enumerate}
\end{proposition}

%% file: appendixB.tex
\section{Appendix on quadrature schemes}
\label{sec:appendix-quadrature}

A quadrature scheme $\{(Y_h,\w_h)\}_{h>0}$
that converges for all continuous functions,
i.e.~such that
\[
\lim_{h \to 0^+} \; \sum_{i=1}^{\abs{Y_h}} w_{h,i} f(y_{h,i})
= \int_{\Omega} f(y) \dy
\quad \text{for all $f \in C^0(\Omegabar)$}
\]
is necessarily stable,
in the sense that there exists a constant
$C_Q > 0$ such that
\[
\norm{\w_h}_1
\deq \sum_{i=1}^{\abs{Y_h}} \abs{w_{h,i}}
\leq C_Q
\quad \text{for all $h > 0$}.
\]
This is a well-known consequence of the uniform boundedness principle
\cite[Theorem A.3]{atkinson1997numerical}
applied to the family $\{\Qcal^0_h\}_{h>0}$
of bounded linear operators
\begin{equation} \label{eq:Qcal0h}
\Qcal^0_h \colon C^0(\Omegabar) \to \R,
\qquad \Qcal^0_h f = \sum_{i=1}^{\abs{Y_h}} w_{h,i} f(y_{h,i})
\quad \text{for all $f \in C^0(\Omegabar)$},
\end{equation}
because the operator norm $\norm{\Qcal^0_h}$
is equal to $\norm{\w_h}_1$.
The same argument, however, cannot be used to
infer stability of a quadrature scheme
$\{(Y_h,\w_h)\}_{h>0}$ of order $q > 0$
(see Definition~\ref{CQ}),
because the operator norm
of $\Qcal^q_h \colon C^q(\Omegabar) \to \R$
is not $\norm{\w_h}_1$ anymore, and can actually be
much smaller, as implied by the following example where it is shown that there
exist quadrature schemes of order $q > 0$ that are
not stable.

\begin{example} \label{ex:unstableqf}
The basic idea is to modify a stable quadrature
scheme of order $q > 0$ to make it unstable without
affecting its order of convergence.
On the closed interval $\Omegabar = [0,1]$,
consider the (right) rectangle rule $(Y_h,\w_h)$ defined
over $N \deq \left\lceil h^{-1} \right\rceil$
uniform subintervals:
\[
Y_h = \{y_i\}_{i=1}^N,
\quad y_i = i/N,
\quad w_i = 1/N.
\]
The quadrature scheme $\{(Y_h,\w_h)\}_{h > 0}$
is stable, because all weights are positive
and sum to one, and is known to have order $q = 1$,
i.e.~there exists $C_w > 0$ such that
\[
\abs{\int_0^1 f(y) \dy
- \sum_{i=1}^{N} \frac{1}{N} \: f\left(\frac{i}{N}\right)}
\leq C_w N^{-1} \norm{f}_{C^1([0,1])}
\quad \text{for all $f \in C^1([0,1])$}.
\]
We can now duplicate each node in the rectangle rule
to make it unstable without altering its order
of convergence. Consider the following
sets of nodes and weights:
\begin{gather*}
Y_h = \{y_i\}_{i=1}^{2N},
\quad y_i = \frac{i}{N} \text{ for $1 \leq i \leq N$},
\quad y_i = \frac{i-N}{N} - \frac{1}{N^2}
    \text{ for $N+1 \leq i \leq 2N$},\\
\w_h = \{w_i\}_{i=1}^{2N},
\quad w_i = 1+\frac{1}{N} \text{ for $1 \leq i \leq N$},
\quad w_i = -1 \text{ for $N+1 \leq i \leq 2N$}.
\end{gather*}
This modified rectangle rule is now unstable,
because $\norm{\w_h}_1 = 2N+1$. However, its rate
of convergence for functions in $C^1$ is not affected:
\begin{align*}
\abs{\int_0^1 f(y) \dy - \sum_{i=1}^{2N} w_i f(y_i)}
&= \abs{\int_0^1 f(y) \dy
    - \sum_{i=1}^{N} \left( 1 + \frac{1}{N} \right)
        f\left(\frac{i}{N}\right)
    - \sum_{i=N+1}^{2N} (-1) \:
        f\left(\frac{i-N}{N} - \frac{1}{N^2}\right)} \\
&\leq \abs{\int_0^1 f(y) \dy
    - \sum_{i=1}^{N} \frac{1}{N} f\left(\frac{i}{N}\right)}
+ \abs{ -\sum_{i=1}^N f\left(\frac{i}{N}\right)
    + \sum_{i=1}^N f\left(\frac{i}{N} - \frac{1}{N^2}\right)} \\
&\leq C_w N^{-1} \norm{f}_{C^1([0,1])}
+ \abs{ -\sum_{i=1}^N f\left(\frac{i}{N}\right)
    + \sum_{i=1}^N f\left(\frac{i}{N} - \frac{1}{N^2}\right)}
\end{align*}
By the mean value theorem, there exists
$\eta_i \in \left( \frac{i}{N} - \frac{1}{N^2},
    \frac{i}{N} \right)$
for all $i=1,\dots,N$ such that
\[
f\left(\frac{i}{N} - \frac{1}{N^2}\right)
= f\left(\frac{i}{N}\right) - f'(\eta_i) \frac{1}{N^2},
\]
from which it follows that
\begin{align*}
\abs{ -\sum_{i=1}^N f\left(\frac{i}{N}\right)
    + \sum_{i=1}^N f\left(\frac{i}{N} - \frac{1}{N^2}\right)}
= \abs{ -\sum_{i=1}^N f\left(\frac{i}{N}\right)
    +\sum_{i=1}^N f\left(\frac{i}{N}\right)
    -\sum_{i=1}^N f'(\eta_i) \frac{1}{N^2}}
\leq N^{-1} \norm{f}_{C^1([0,1])}.
\end{align*}
\end{example}

It follows immediately from this example
that a quadrature scheme of order $q > 0$
is not necessarily convergent for continuous
functions. More generally, it can be shown
that a quadrature scheme of order $q > 0$
might not converge for functions
in $C^{q'}(\Omegabar)$ with $q' < q$.
However, it turns out that the convergence for
continuous functions can be recovered under
the additional assumption of stability.
This is important because convergence for
continuous functions is a key assumption in
Theorem 4.1.2 of \cite{atkinson1997numerical},
on which our proof of
Proposition~\ref{prop:nystrom-classical} is based.

\begin{proposition} \label{prop:convcont}
Let $\{(Y_h,\w_h)\}_{h > 0}$ be a stable quadrature
scheme over a bounded domain $\Omega$ with order
of convergence $q > 0$. Then,
\[
\lim_{h \to 0^+} \; \sum_{i=1}^{\abs{Y_h}} w_{h,i} f(y_{h,i})
= \int_{\Omega} f(y) \dy
\quad \text{for all $f \in C^0(\Omegabar)$}.
\]
\end{proposition}

\begin{proof}
Define the $h$-indexed family of bounded linear
operators $\Qcal^0_h$ as in \eqref{eq:Qcal0h},
and let $\Qcal^0$ be the bounded linear operator
\[
\Qcal^0 \colon C^0(\Omegabar) \to \R,
\qquad
\Qcal^0(f) \deq \int_{\Omega} f(y) \dy.
\]
By the definition of order of convergence of a
quadrature scheme,
\[
\lim_{h \to 0^+} \Qcal^0_h f = \Qcal^0 f
\quad \text{for all $f \in C^q(\Omegabar)$},
\]
and, by Lemma~\ref{lemma:density}, the space
$C^q(\Omegabar)$ is a dense subset of $C^0(\Omegabar)$.
Moreover, the operator norm of $\Qcal^0_h$
is equal to $\norm{\w_h}_1$, and so the stability
of the quadrature scheme is equivalent to the
uniform boundedness of the family $\{\Qcal^0_h\}_{h > 0}$.
By Lemma~\ref{lemma:strong-convergence} and the standard
sequential-criterion reduction (testing every sequence
$h_i \to 0^+$), the operators $\Qcal^0_h$ converge strongly
to $\Qcal^0$ as $h \to 0^+$, which means that
\[
\lim_{h \to 0^+} \Qcal^0_h f = \Qcal^0 f
\quad \text{for all $f \in C^0(\Omegabar)$}. \qedhere
\]
\end{proof}